\documentclass[11pt]{article}  
\usepackage[twoside, a4paper, lmargin=3.5cm, rmargin=3.5cm, top=2.5cm, bottom=2.5cm]{geometry} 

\usepackage{amsmath}
\usepackage{amssymb}
\usepackage{euscript}

\usepackage{amsmath}
\usepackage{amssymb}
\usepackage{euscript}
\usepackage{graphicx}
\usepackage{color}
\definecolor{dgreen}{rgb}{0,.6,0}

\newtheorem{theorem}{\bf Theorem}[section]
\newtheorem{lemma}[theorem]{\bf Lemma}

\newtheorem{corollary}[theorem]{\bf Corollary}

\newenvironment{proof}{\noindent{\em Proof:}}{\quad \hfill$\Box$\vspace{2ex}}

\def\ZZ{{\mathbb Z}}

\def\II{{\mathbb I}}

\def\FF{{\mathbb F}}
\def\NN{{\mathbb N}}
\def\RR{{\mathbb R}}

\def\TT{{\mathbb T}}

\def\TTm{{\mathbb T}^m}

\def\ZZm{{\mathbb Z}^m}
\def\IIm{{\mathbb I}^m}
\def\ZZmp{{\mathbb Z}^m_+}
\def\ZZms{{\mathbb Z}^m_*}
\def\IIi{{\mathbb I}^\infty}

\def\RRi{{\mathbb R}^\infty}
\def\TTi{{\mathbb T}^\infty}

\def\RRi{{\mathbb R}^\infty}

\def\ZZi{{\mathbb Z}^{\infty}}
\def\ZZi{{\mathbb Z}^{\infty}}
\def\ZZis{{\mathbb Z}^{\infty}_*}
\def\ZZip{{\mathbb Z}^\infty_+}

\def\Ii{{\mathcal I}}
\def\Jj{{\mathcal J}}

\def\Ll{{\mathcal L}}
\def\Pp{{\mathcal P}}
\def\Ss{{\mathcal S}}
\def\Tt{{\mathcal T}}
\def\Uu{{\mathcal U}}
\def\Vv{{\mathcal V}}

\def\supp{\operatorname{supp}}

\newcommand{\ba}{{\bf a}}
\newcommand{\bb}{{\bf b}}
\newcommand{\bc}{{\bf c}}
\newcommand{\bd}{{\bf d}}

\newcommand{\bk}{{\bf k}}

\newcommand{\bs}{{\bf s}}

\newcommand{\bx}{{\bf x}}
\newcommand{\by}{{\bf y}}
\newcommand{\rep}[1]{\tilde{#1}}
\def\LiD{L_\infty(\TT^m)}

\def\Wi{W^1_\infty(\TT^m)}

\newcommand{\Id}{\operatorname{Id}}

\newlength{\fixboxwidth}
\title{$\varepsilon$-dimension in infinite dimensional hyperbolic cross approximation and application to parametric elliptic PDEs}
\author{ 
Dinh D\~ung \footnote{Corresponding author at: Information Technology Institute, Vietnam National University, Hanoi, 
144 Xuan Thuy, Cau Giay, Hanoi, Vietnam.} \\[2mm]
 Information Technology Institute, Vietnam National University, Hanoi, \\
144 Xuan Thuy, Cau Giay, Hanoi, Vietnam\\
{dinhzung@gmail.com}\\[5mm] 
Michael Griebel  \\[2mm]
Institute for Numerical Simulation, Bonn University \\
Wegelerstrasse 6, 53115 Bonn, Germany \\
Fraunhofer Institute for Algorithms and Scientific Computing SCAI\\
Schloss Birlinghoven, 53754 Sankt Augustin, Germany \\
{griebel@ins.uni-bonn.de}\\[5mm] 
 Vu Nhat Huy \\[2mm]
College of Science,  Vietnam National University, Hanoi\\
334 Nguyen Trai, Thanh Xuan, Hanoi, Vietnam\\
{nhat{\_}huy85@yahoo.com} \\[5mm]
Christian Rieger \\[2mm]
Institute for Numerical Simulation, Bonn University \\
Wegelerstrasse 6, 53115 Bonn, Germany \\
{rieger@ins.uni-bonn.de}
}
\date{\ttfamily  July 07, 2016 - Version 1.1}
\begin{document}
\maketitle
\begin{abstract}
In this article, we present a cost-benefit analysis of the approximation in tensor products of Hilbert spaces of Sobolev-analytic type. 
The Sobolev part is defined on a finite dimensional domain, whereas the analytical space is defined on an infinite dimensional domain. 
As main mathematical tool, we use the $\varepsilon$-dimension of a subset in a Hilbert space. The $\varepsilon$-dimension gives the lowest number of linear information that is needed to approximate an element from the set in the norm of the Hilbert space up to an accuracy $\varepsilon>0$. From a practical point of view this means that we a priori fix an accuracy and ask for the amount of information to achieve this accuracy.  Such an analysis usually requires sharp estimates on the cardinality of certain index sets which are in our case infinite-dimensional hyperbolic crosses. 
As main result, we obtain sharp bounds of the $\varepsilon$-dimension of the Sobolev-analytic-type function classes which depend only on the smoothness differences in the Sobolev spaces and the dimension of the finite dimensional domain where these spaces are defined.
This implies in particular that, up to constants, the costs of the infinite dimensional (analytical) approximation problem is dominated by the finite-variate Sobolev approximation problem. We demonstrate this procedure with an examples of functions spaces stemming from the regularity theory of parametric partial differential equations.

\medskip
\noindent
{\bf Keywords}: infinite-dimensional hyperbolic cross approximation, mixed Sobolev-Korobov-type smoothness, mixed Sobolev-analytic-type smoothness, $\varepsilon$-dimension, parametric and stochastic elliptic PDEs, linear information.

\end{abstract}

\section{Introduction}

The main emphasis of this paper lies on the cost-benefit ratio of the approximation for a class of functions stemming from an anisotropic tensor product of smoothness spaces. 
Let $X$ be a Hilbert space and $W \subset X$ a subset of $X$. Since we are interested in the cost-benefit ratio of the approximation, we focus on the so-called $\varepsilon$-dimension $n_\varepsilon = n_\varepsilon(W,X)$.  
It is defined as 
\begin{align}\nonumber
n_\varepsilon(W,X)
:= \ 
{
\inf \left\{n:\, \exists  \ M_n: 
\ \dim M_n \le n, \ \sup_{w \in W} \ \inf_{v \in M_n} \|w - v\|_X \le \varepsilon \right\}
},
\end{align}
where $M_n \subset X$ is a linear manifold in $X$ of dimension $\le n$.
Hence, $n_\varepsilon(W,X)$ is the smallest number of linear functionals that are needed by an algorithm to give for all $f \in W$ an approximation with an error of at most $\varepsilon$.
The important concept here is the fact that an approximation quality $\varepsilon>0$ is a priori fixed  and the approximation space realizing this approximation error is searched. 
This is the inverse of the usual Kolmogorov $n$-width $d_n(W,X)$ \cite{Ko36} which is given by
\begin{equation} \nonumber
d_n(W,X)
:= \ 
\inf_{{M_n}} \ \sup_{w \in W} \ \inf_{v \in {M_n}} \|w - v\|_X,
\end{equation} 
where the outer infimum is taken over all linear manifolds {$M_n$} in $X$ of dimension at most $n$.
\footnote{A different worst-case setting is represented by the linear $n$-width $\lambda_n(W,X)$ \cite{Ti60}.
 This corresponds to a characterization of the best linear approximation error, see, e.g., \cite{DU13} for definitions.
 Since $X$ is here a Hilbert space, both concepts coincide, i.e., we have
 $d_n(W,X) \ = \ \lambda_n(W,X).$
 }
For a survey and a bibliography on computational complexity see the monographs \cite{NW08, NW10}.
 
To be more specific, we deal with functions defined on a product domain $\TTm \times \IIi$, where $\IIi$ is infinite dimensional and $\TTm$ is $m<\infty$ dimensional.
The fundamental space is defined 
\begin{equation}\nonumber
\Ll:=\left\{v \in \Ll \ : \ v := \sum_{(\bk,\bs) \in \Ii \times \Jj} v_{\bk,\bs} \phi_{\bk,\bs} \quad \text{such that} \quad \sum_{(\bk,\bs) \in \Ii \times \Jj} |v_{\bk,\bs}|^2<\infty \right\},
\end{equation}
where $\phi_{\bk,\bs}$ denotes an orthonormal system with respect to the inner product
\begin{equation} \nonumber
\left( v,w\right)_{\Ll}=\sum_{(\bk,\bs) \in \Ii \times \Jj} v_{\bk,\bs} \bar{w}_{\bk,\bs}.
\end{equation}
In order to study approximation numbers such as $n_\varepsilon(W,X)$, we need to define 
the smoothness space $X$ and the smoothness class $W$ as well. Smoothness spaces are modeled here by general sequences of scalars 
$\lambda := \{\lambda(\bk,\bs)\}_{(\bk,\bs) \in \Ii \times \Jj}$ with 
$\lambda(\bk,\bs) \not= 0$.
Then, we define the associated space (see \eqref{Llambda}
\begin{equation}\label{spacelambda}
\Ll^\lambda:=\left\{v \in \Ll \ : \ \text{there exists } \rep{v} \in \Ll \text{ such that } v := \sum_{(\bk,\bs) \in \Ii \times \Jj} \frac{\rep{v}_{\bk,\bs}}{\lambda(\bk,\bs)} \, \phi_{\bk,\bs}\right\}
\end{equation}
The norm on $\Ll^\lambda$ is defined by (see \eqref{P-Id[norm-Ll^lambda]}
\begin{equation*}
\|v\|^{2}_{\Ll^\lambda} \ := \|\rep{v}\|^{2}_\Ll=\sum_{(\bk,\bs) \in \Ii \times \Jj} |\lambda(\bk,\bs)|^2 \,  |v_{\bk,\bs}|^2,
\end{equation*}
where $\tilde{v}$ is defined in \eqref{spacelambda}.
Let us assume to have two such sequences $\lambda$ and $\nu$ with $\nu \le \lambda$ in a point-wise sense.
Then we can chose 
\begin{align*}
V=\Ll^{\nu} \quad \text{and} \quad W=\Uu^{\lambda},
\end{align*}
 where $\Uu^\lambda$ denotes the unit ball in $\Ll^\lambda$.
Hence, we are left with estimating $n_{\varepsilon}(\Uu^{\lambda},\Ll^{\nu})$. 
To account for the fact that we work on a  product domain $\TTm \times \IIi$, the concrete smoothness spaces are parametrized by a number $a$ and a 
sequence $\bb$ such that 
$\rho_{a,\bb}(\bk,\bs)$ are product and order dependent weights (see also \eqref{[lambda{a,mu,br}]})
\begin{equation}\label{introrho}
\rho_{a,\bb}(\bk,\bs)
 := \
 \max_{1\le j \le m}|k_j|^a \, \frac{\bs!}{|\bs|_1!} \bb^{-\bs}.
\end{equation}
Both $\lambda$ and $\nu$ will be of this specific form. We provide a motivation for such classes of functions spaces by considering the regularity spaces arising in the theory of parametric partial differential equations (PDEs).
The simpler case of tensor product weights
\begin{equation*}
\rho_{a,\bb}(\bk,\bs)
 := \
 \max_{1\le j \le m}|k_j|^a \,  \bb^{-\bs}
\end{equation*}
was already treated in \cite{DG16}. 
The main result of this paper is the fact that the $\varepsilon$-dimension of our Sobolev-analytic-type function 
class depends only on the smoothness differences in the finite-variate Sobolev spaces and the dimension of the finite dimensional domain where these spaces are defined.
This implies in particular that, up to constants, the costs of solving the infinite dimensional (analytical) approximation problem are dominated by the finite-variate Sobolev-smooth approximation problem.

\bigskip
The remainder of the paper is organized as follows:
In Section \ref{sec:PDE}, we consider the general parametrized elliptic Poisson problem and its regularity results both with respect to the spatial and with respect to the infinite-dimensional parametric component.  
In Section \ref{sec:infiniteapprox}, we review the setting of infinite dimensional tensor products of Hilbert spaces   and the associated approximation and $\varepsilon$-dimension. 
In Section \ref{solution joint regularity}, we give more details on the applications of the general setting to the smoothness spaces arising in parametric PDEs. 
The main mathematical results concern the cardinality of the infinite dimensional hyperbolic crosses in Section \ref{cardinality of HC}.
This section is split into two steps. The first result in \ref{condition for summability} addresses the inclusion $\Big(\frac{|\bs|_1!}{\bs!}\bb^\bs\Big)_{\bs \in \FF} \in \ell_p(\FF)$ with $0 < p < \infty$. This is in particular novel since the case $p>1$ is included here. The main result in this section is Theorem \ref{theorem[<E<]} which is proven based on a result in Section \ref{Estimates of cardinality of HC}.
Here, the summability condition enters an absolute constant. 
In Section \ref{finalrates},
we combine our results to derive sharp estimates 
 of the $\varepsilon$-dimension and its inverse, the Kolmogrov $n$-widths of the Sobolev-analytic-type function classes. These results are then applied to the Galerkin approximation of parametric elliptic PDEs. 
 We finish the paper with some concluding remarks in Section \ref{conclusion}.   

\bigskip
\noindent
{\bf Notation.}
We will use the following notation: $\ZZms := \{\bk \in \ZZm: k_j \not= 0, \ j =1,...,m\}$; $\RRi$ is the set of all sequences $\by= (y_j)_{j=1}^\infty$ with $y_j \in \RR$; $|\bk|_\infty := \max_{1 \le j \le m}|k_j|$ for $\bk \in \ZZm$.
Similarly, we set $\II=[-1,1]$ and $\IIi$ is the set of all sequences $\by= (y_j)_{j=1}^\infty$ with $y_j \in \II$.
$\ZZi$ is the set of all sequences $\bs= (s_j)_{j=1}^\infty$ with $s_j \in \ZZ$. 
Furthermore, 
$\ZZip := \{\bs \in \ZZi: s_j \ge 0, \ j =1,2,...\}$, $y_j$ is the $j$th coordinate of 
$\by \in \RR^\infty$. Moreover, $\FF$ is a subset of $\ZZip$ of all $\bs$ such that $\supp(\bs)$ is finite, where 
$\supp(\bs)$ is the support of $\bs$, that is the set of all $j \in \NN$ such that 
$s_j \not=0$.
If $\bs \in \FF$, we define
\[
\bs!:= \prod_{j=1}^\infty s_j!,  \quad |\bs|_1:= \sum_{j=1}^\infty s_j, \quad \text{and} \quad
\bb^\bs
:= \
\prod_{j=1}^\infty  b_j^{s_j}
\] 
for a sequence  $\bb=(b_j)_{j \in \NN}$ of positive numbers.

\section{Parametric Operator equations}
\label{sec:PDE}
Let us briefly recall the setting of \cite{KS16}. Denote by $X$ a real separable Banach space over the field $\RR$ and by $X^{\prime}$ its topological dual, i.e., the bounded linear functionals. We consider a map
\begin{align*}
\mathcal{G}: \left(\IIi,\left\|\cdot \right\|_{\infty} \right)\to \mathcal{L}_{I}(X,X^{\prime}), \quad \by \mapsto \mathcal{G}(\by)=G_{\by},
\end{align*}
where $\mathcal{L}_{I}$ denotes the space of boundedly invertible linear operators. By $G_{\by}^{-1}\in\mathcal{L}_{I}(X,X^{\prime})$, we denote the element such that $G_{\by} \circ G^{-1}_{\by}=\Id_{X^{\prime}}$ and $G^{-1}_{\by} \circ G_{\by}=\Id_{X}$.
We define 
\begin{align*}
\mathcal{G}_{-1}:\left(\IIi,\left\|\cdot \right\|_{\infty} \right) \to \mathcal{L}_{I}(X,X^{\prime}), \quad \by \mapsto \mathcal{G}_{-1}(\by)=G^{-1}_{\by}.
\end{align*}
We assume that $\mathcal{G}_{-1}$ is bounded by $C(\mathcal{G})$ i.e., that
\begin{align}\label{opbounded}
\sup_{\genfrac{}{}{0pt}{}{\by \in \IIi}{\|\by \|_{\infty}\le 1}} \left\|\mathcal{G}_{-1}(\by) \right\|_{\mathcal{L}(X,X^{\prime})}=\sup_{\by \in \IIi} \left\|G^{-1}_{\by} \right\|_{\mathcal{L}(X,X^{\prime})}\le C(\mathcal{G}).
\end{align} 
Moreover, we assume that $\mathcal{G}$ is analytic with respect to every $y_{j}$ with $j \in \NN$ and that there is a sequence $\bd:\NN \to \RR$ with $\bd \in \ell^{p}(\NN)$ for a fixed $0<p\le 1$ such that for all $\bs \in \FF\setminus \{0\}$
\begin{align}\label{anacond}
\sup_{\by \in \IIi} \left\|\mathcal{G}_{-1}(\mathbf{0}) \partial_{\by}^{\bs} \mathcal{G}(\by) \right\|_{\mathcal{L}(X,X^{\prime})} \le C(\mathcal{G}) \bd^{\bs} .
\end{align}
Furthermore, we observe that we can write the solution $u \in X$ of the operator equation $G(\by)u(\by)=f$ for given $f \in X^{\prime}$ in terms of the solution operator
\begin{align}\nonumber
\mathcal{S}:\IIi \times X^{\prime} \to X, \quad (\by,f) \mapsto \mathcal{S}(\by,f):=\mathcal{G}_{-1}(\by) f = G^{-1}(\by)f
\end{align}
and \cite[Thm.~4]{KS16} provides the bound
\begin{align}\label{solmapbound}
\sup_{\by \in \IIi}\sup_{\genfrac{}{}{0pt}{}{f \in X^{\prime}}{\|f\|_{X^{\prime}}=1}} \left\|\partial_{\by}^{\bs}\mathcal{S}(\by,f) \right\|_{X} \le \frac{C(\mathcal{G})}{\ln(2)}\left|\bs\right|! \bd^{\bs} 
\end{align}
for all $\bs \in \FF\setminus \{0\}$. This implies a (generalized) Taylor's series representation of
\begin{align}\nonumber
u(\by)=u_{f}(\by)= \sum_{s \in \FF} \frac{1}{\bs!} \partial_{\by}^{\bs}u_{f}(\by) \by^{\bs}=\sum_{s \in \FF} \left(\frac{1}{\bs!} \left.\partial_{\by}^{\bs}u_{f}(\by)\right|_{\by=0}\right)\by^{\bs}=\sum_{s \in \FF} \left(\frac{1}{\bs!} \left.\partial_{\by}^{\bs} \mathcal{S}(\by,f)\right|_{\by=0}\right)\by^{\bs}
\end{align} 
Hence, the coefficient are bounded by
\begin{align} \label{partial_yS}
\left|\frac{1}{\bs!} \left.\partial_{\by}^{\bs} \mathcal{S}(\by,f)\right|_{\by=0}\right| \le \frac{C(\mathcal{G})}{\ln(2)}\frac{\left|\bs\right|!}{\bs!} \bd^{\bs} \left\|f\right\|_{X^{\prime}},
\end{align}
which fits exactly into our framework, i.e., the upper bound has the structure of $\rho_{a,\bb}^{-1}(\bk,\bs)$ with $a=0$ from \eqref{introrho}. We will, however, study a more specific example in more detail, since we also need spatial regularity results, which allows also for $a>0$. 
For the elliptic PDEs \eqref{SPDE} formulated in the next section, some particular estimates for the coefficients in the Taylor and Legendre expansions which are similar to \eqref{solmapbound} and \eqref{partial_yS} were established in earlier papers \cite{BNTT,CDS10b,CDS10a}.

\subsection{Parametric elliptic PDEs}
Here, we consider a more specific problem which fits into the framework outlined above.
We chose $X=H_{0}^{1}(\IIm)$ and hence $X^{\prime}=H^{-1}(\IIm)$. The operator is 
\begin{align*}
\mathcal{G}_{a}:\IIi \to \mathcal{L}(H_{0}^{1}(\IIm),H^{-1}(\IIm)), \quad \by \mapsto \left(H_{0}^{1}(\IIm) \ni u \mapsto - \operatorname{div} (a(\by)\nabla u(\by)) \in H^{-1}(\IIm)\right),
\end{align*}
where $a:\IIi \times \IIm \to \RR_{+}$ is a function satisfying
\begin{equation} \nonumber
\quad 0 \ < \  r \ < \ a(\by,\bx)  \le \ R \ < \ \infty, \quad \bx \in \IIm, \ \by \in \IIi.
\end{equation}
In order to derive spatial regularity, we will restrict ourselves to $f \in L^{2}(\IIm) \subset H^{-1}(\IIm)$.
Moreover, we restrict ourselves to periodic problems, that is $a(\by)(\bx):=a(\bx,\by)$  is a function of $\bx=(x_1,...,x_m) \in \II^m$ and of parameters $\by=(y_1,y_2,...) \in \IIi$ on $\II^m \times \IIi$, and the function $f(\bx)$ is a function of $\bx=(x_1,...,x_m) \in \II^m$. We will assume that $a(\by)$ and $f$ as functions on $\bx$ can be extended to $1$-periodic functions in each variable $x_j$ on the whole $\RR^m$, and hence  $a(\by)$ and $f$ can be considered as functions defined om $\TTm$. 
Hence, we consider the parametric elliptic problem 
\begin{equation} \label{SPDE}
- \operatorname{div} (a(\by)\nabla_{x} u(\by))
\ = \
f \quad \text{in} \quad \II^m,
\quad u|_{\partial \II^m} \ = \ 0, \quad \by \in \IIi.
\end{equation}
Throughout the present paper we also preliminarily assume that $f \in H^{-1}(\IIm)$ and  the diffusions $a$ satisfy the {\em uniform ellipticity assumption} which ensures condition \eqref{opbounded}
\begin{equation} \nonumber
\quad 0 \ < \  r \ < \ a(\by)(\bx)=a(\bx,\by) \ \le \ R \ < \ \infty, \quad \bx \in \TTm, \ \by \in \IIi.
\end{equation}

Let $V:= H^1_0(\TTm)$ and  denote by $W$ the subspace of $V$ equipped with the semi-norm and norm
\begin{equation} \nonumber
|v|_W 
:= \
\|\Delta v\|_{L_2(\TTm)}, \quad 
\|v\|_W 
:= \
\|v\|_V + |v|_W.
\end{equation}

Note that if $v \in L_2(\TTm)$ and 
\begin{equation} \nonumber
v 
\ = \
\sum_{\bk \in \ZZm} v_\bk e_\bk,
\end{equation}
where $e_\bk(x):=  e^{i2\pi \bk x}$, i.e., $\{e_\bk\}_{\bk \in \ZZ^{m}}$ is the usual orthonormal basis of 
$L_2(\TT^{m})$,
then from the definition and Parseval's identity we have
\begin{equation} \label{Parseval-V}
(2\pi)^2\sum_{\bk \in \ZZms} |\bk|_\infty^2 |v_\bk|^2
\ \le \
\|v\|_V^2 = \sum_{i=1}^{m}\|\partial_{i} v\|^{2}_{L^{2}(\TT^{m})}=(2\pi)^2\sum_{\bk \in \ZZms} |\bk|_{2}^2 |v_\bk|^2
\ \le \
(2\pi)^{2}m\sum_{\bk \in \ZZms} |\bk|_\infty^2 |v_\bk|^2,
\end{equation}
and 
\begin{equation} \label{Parseval-W}
(2\pi)^{4}\sum_{\bk \in \ZZms} |\bk|_\infty^4 |v_\bk|^2
\ \le \
|v|_W^2 
\ \le \ (2\pi)^{4}m^{2}\sum_{\bk \in \ZZms} |\bk|_\infty^4 |v_\bk|^2,
\end{equation}
where we used the norm equivalence $|\bk|_{\infty}\le |\bk|_{2} \le \sqrt{m}|\bk|_{\infty}$ for all $\bk \in \ZZms$.
 
\subsection{Spatial regularity}\label{sec:spatialreg}
By the well-known Lax-Milgram lemma, there exists a unique (weak) solution $u \in V$ to equation \eqref{SPDE} which satisfies the variational equation
\begin{equation} \nonumber
\int_{\TTm} a(\bx,\by) \nabla u(\bx,\by) \cdot \nabla v(\bx) \, \mbox{d}\bx
\ = \
\int_{\TTm} f(\bx) \, v(\bx) \, \mbox{d}\bx, \quad \forall v \in V.
\end{equation}
We skip the explicit dependence on the parameter $\by$ in this section. 
 Moreover, this solution satisfies the inequality
\begin{equation} \nonumber
\|u\|_V 
\ \le \
\frac{\|f\|_{V^*}}{r},
\end{equation}
where $V^* = H^{-1}(\TTm)$ denotes the dual of $V$. 
Observe that there holds the embedding $L_2(\TTm) \hookrightarrow V^*$ and the  inequality
\begin{equation} \nonumber
\|f\|_{V^*} 
\ \le \
\|f\|_{L_2(\TTm)}.
\end{equation}

If we assume that $a \in \Wi$, then the solution $u$ of \eqref{SPDE} is in $W$. Moreover, $u$ satisfies the estimates
\begin{equation} \nonumber
|u|_W 
\ \le \
\frac{1}{r}\left(1 + \frac{|a|_{\Wi}}{r}\right) \|f\|_{L_2(\TTm)},
\end{equation}
and
\begin{equation} \nonumber
\|u\|_W 
\ \le \
\frac{1}{r}\left[1 + \left(1 + \frac{|a|_{\Wi}}{r}\right)\right] \|f\|_{L_2(\TTm)}.
\end{equation}
This spatial regularity implies certain approximation rate if we use trigonometric polynomials in a Galerkin approach. 
For a real positive number $T \geq 1$ we define 
the index set 
\begin{align*}
G_{\ZZms}(T):=\{\bk \in \ZZms: \ |\bk|_\infty \le T\}.
\end{align*}
Denote by $\Tt_n$ with $n=(2\lfloor T\rfloor)^{m}=|G_{\ZZms}(T)|$ the space of trigonometric polynomials
\begin{equation} \nonumber
\Tt_n:= \left\{ v : \sum_{\bk \in G_{\ZZms}(T)} v_\bk e_\bk \right\}
\end{equation}
of dimension $n$. Let $P_n$ be the projection from $L_2(\TTm)$ onto $\Tt_n$. Then, we get using $2^{-1}n^{1/m}=\lfloor T \rfloor \le T \le 2^{-1}n^{1/m}+1$ and $T\ge 1$ that
\begin{align*}
\|u - P_n(u)\|_V &\le 2\pi \left(m \sum_{\bk \in \ZZms \setminus G_{\ZZms}(T)} |\bk|_\infty^2 |v_\bk|^2\right)^{\frac{1}{2}}  \le 2\pi \sqrt{m} \left( \sum_{\bk \in \ZZms \setminus G_{\ZZms}(T)} T^{-2}|\bk|_\infty^4 |v_\bk|^2\right)^{\frac{1}{2}}\\&\le 2\pi \sqrt{m} T^{-1}|u|_W \le  4\pi \sqrt{m} n^{-\frac{1}{m}}|u|_W
\end{align*}
holds for all $u \in W$. Furthermore, we obtain $n_{\varepsilon}(V,W) \lesssim G_{\ZZms}(\varepsilon^{-1})$ for $0<\varepsilon\le 1$.
Let $u_n$ be the Galerkin approximation, i.e., the unique solution of the problem
\begin{equation} \nonumber
\int_{\TTm} a(\bx,\by) \nabla u_n(\bx,\by) \cdot \nabla v(\bx) \, \mbox{d}\bx
\ = \
\int_{\TTm} f(\bx) \, v(\bx) \, \mbox{d}\bx, \quad \forall v \in \Tt_n.
\end{equation}
Then, we get with \eqref{Parseval-V}, \eqref{Parseval-W}, and with C\'ea's lemma that
\begin{equation} \nonumber
\|u - u_n\|_V
\ \le \ 
\sqrt{\frac{R}{r}} \, \inf_{v \in \Tt_n}\|u - v\|_V
\ = \ 
\sqrt{\frac{R}{r}}\, \|u - P_n(u)\|_V
\ \le \ 
\sqrt{\frac{R}{r}}\, 4\pi \sqrt{m} n^{-1/m}|u|_W
\ \le \ 
C \, n^{-1/m},
\end{equation}
where we can explicitly compute the constant to be
\begin{equation} \nonumber
C
:= \
2 \sqrt{\pi}\sqrt{\frac{m R}{r}}\,\frac{1}{r}\left(1 + \frac{|a|_{\Wi}}{r}\right)\|f\|_{L_2(\TTm)}.
\end{equation}

\subsection{Parametric regularity}\label{sec:paramreg}

A probability measure on $\IIi$  is the
infinite tensor product measure $\mu$ of the univariate uniform probability measures on the one-dimensional $\II$, i.e.
\[
\mathrm{d} \mu(\by) 
\ = \ 
\bigotimes_{j \in \ZZ} \frac{1}{2}\mathrm{d} y_j.
\]
Here, the sigma algebra $\Sigma$ for $\mu$ is generated by the finite rectangles
$ 
\prod_{j \in \NN} I_j,
$
where only
a finite number of the $I_j$ are different from $\II$ and those that are different are intervals
contained in $\II$. Then, $(\IIi, \Sigma,  \mu)$ is a probability space.

Now, let $L_2(\IIi, \mu)$ denote the Hilbert space of functions on $\IIi$ equipped
with the inner product
\[
\langle v, w \rangle 
:= \
\int_{\IIi} v(\by) \overline{w(\by)} \, \mathrm{d} \mu(\by).
\]
The norm in $L_2(\IIi, \mu)$ is defined as $\|v\| := \langle v,v \rangle^{1/2}$. In what follows, 
$\mu$ is fixed, and, for convention, we write $L_2(\IIi, \mu):= L_2(\IIi)$. Furthermore, let $L_2(\TTm)$ be the usual Hilbert space of Lebesgue square-integrable functions on $\TTm$ based on the univariate 
normed Lebesgue measure. 
Then, we define 
\begin{equation}\nonumber
L_2(\TTm \times \IIi)
:= \
L_2(\TTm) \otimes L_2(\IIi).
\end{equation}
The space $L_2(\TTm \times \II^s) = L_2(\TTm) \otimes L_2(\II^s)$ can be considered as a subspace of 
$L_2(\TTm \times \IIi)$.

Let us reformulate  the parametric equation \eqref{SPDE} in the variational form. 
 For every $\by \in \IIi$, by the well-known Lax-Milgram lemma, there exists a unique solution 
$u(\by) \in V$ in weak form which satisfies the variational equation
\begin{equation} \nonumber
\int_{\TTm} a(\bx,\by)\nabla u(\by)(\bx) \cdot \nabla v(\bx) \, \mbox{d}\bx
\ = \
\int_{\TTm} f(\bx) \, v(\bx) \, \mbox{d}\bx, \quad \forall v \in V.
\end{equation}
Moreover, $u(\by)$ satisfies the estimate
\begin{equation} \nonumber
\|u(\by)\|_V 
\ \le \
\frac{\|f\|_{V^*}}{r}, \ \forall \by \in \IIi.
\end{equation}
Therefore, from the inclusions $u \in L_\infty(\IIi,V) \subset L_2(\IIi,V,\mu)$ it follows that 
$u$ admits the unique expansions 
\begin{equation} \label{Legendre-series}
u
 = \
\sum_{\bs \in \FF} u_\bs \, L_\bs,
\end{equation} 
where 
$\{L_s\}_{s=0}^\infty$ are the family of univariate orthonormal Legendre polynomials in 
$L_2(\II)$ and 
\[
L_\bs(\by):= \ \prod_{j \in \supp(\bs)} L_{s_j}(y_j).
\]
The expansion \eqref{Legendre-series} for $u$ converges in $L_2(\IIi,V,\mu)$, where the Legendre coefficients 
$u_\bs \in V$ are defined by
\[
u_\bs
:= \
\langle u, L_\bs \rangle 
:= \
\int_{\IIi}u(\by) L_{\bs}(\by) \, \mathrm{d} \mu(\by) \quad  
\quad \bs \in \FF.
\]

From \cite[Theorem 2.1]{BNTT} (or from the more general bound \eqref{solmapbound} for the parametric elliptic PDEs \eqref{SPDE}) and the formulas for the Legendre coefficients 
\begin{equation} \nonumber
u_\bs
\ = \
\frac{1}{\bs!} \, \prod_{j: \ s_j \not=0}\frac{2s_j + 1}{2^{s_j}}
\int_{\IIi} \partial^\bs_\by u(\by) \prod_{j: \ s_j \not=0} (1 - y_j^2)^{s_j} \mbox{d} \mu(\by)
\end{equation}
we derive the following result.

\begin{lemma} \label{lemma|u_bs|_V}
Assume that the diffusions  $a$ are infinitely times differentiable  with respect to $\by$ and that there exists a positive sequence $\ba = \ (a_j)_{j \in \NN}$ such that  
\begin{equation} \nonumber
\|\partial^\bs_\by a(\by)\|_V
\ \le \
\ba^\bs
 \quad  \by \in \IIi, \quad \bs \in \FF.
\end{equation}  
Then we have 
\begin{equation} \nonumber
\|u_\bs\|_V 
\ \le \
K \frac{|\bs|!}{\bs!}\, \bd^\bs, \quad \bs \in \FF,
\end{equation}
where $K:= \ \frac{\|f\|_{V'}}{r}$ and $\bd := \ \frac{\ba}{\ln 2}$.
\end{lemma}

Now, denote by $\Wi$ the space of functions $v$ on $\TTm$, equipped with the semi-norm and the norm 
\[ 
|v|_{\Wi}
:= \
\max_{1 \le i \le m} \|\partial_{x_i} v\|_{\LiD}, \quad 
\|v\|_{\Wi}
:= \
\|v\|_{\LiD} + |v|_{\Wi}
\]
respectively.
For the proof of the following lemma see \cite[Lemma 5.5]{Di16}.

\begin{lemma} \label{lemma|u_bs|_W}
Assume that  $f \in L_2(\II^m)$, assume that the diffusions  $a \in L_\infty(\IIi,\Wi)$ and that they are affinely dependent with respect to $\by$ as 
\begin{equation} \label{KL-exp}
a(\by)(\bx)
\ = \
\overline{a}(x) \ + \ 
\sum_{j=1}^\infty  y_j\, \psi_j(\bx), \quad \bx \in \TTm, \ y \in \IIi, \quad \overline{a},\psi_j \in \Wi.
\end{equation}
Then we have that
\begin{equation}\nonumber
\|u_\bs\|_W 
\ \le \
K \frac{|\bs|!}{\bs!}\, \bd^\bs, \quad \bs \in \FF,
\end{equation}
where 
\begin{equation} \nonumber
K:= \
\frac{1}{r}\left(1 + \left(1 + \frac{|a|_{L_\infty(\IIi,\Wi)}}{r}\right)\right) \|f\|_{L_2(\TTm)}
\end{equation}
and
\begin{equation} \nonumber
\bd \ = \ (d_j)_{j \in \NN}, \quad
d_j 
:= \
\frac{1}{r\sqrt{3}} \left(\left(\frac{|a|_{L_\infty(\IIi,\Wi)}}{r} 
+ 2\right)\|\psi_j\|_{L_\infty(\TTm)} + |\psi_j|_{\Wi}\right).
\end{equation}
\end{lemma}
The affine structure in \eqref{KL-exp} makes it easy to check the condition \eqref{anacond}. 
Furthermore, see \cite[Section 2.3]{SG} for more details where the setting of general operator equations includes parametric elliptic PDEs as special case.

We will see in Section \ref{solution joint regularity} that the spatial and parametric regularities of the solution $u$ to \eqref{SPDE} induce a joint regularity in infinite tensor product Hilbert spaces which is appropriate to hyperbolic cross approximation in infinite dimension.

\section{Approximation in infinite tensor product Hilbert spaces of joint regularity}
\label{sec:infiniteapprox}

In this section, we recall some results on approximation in infinite tensor product Hilbert spaces of joint regularity
which were proven in \cite[Subsection 3.1]{DG16}.
We first introduce the notion of the infinite tensor product of separable Hilbert spaces. 
Let $H_j$, $j=1,...,m$, be separable Hilbert spaces with inner products $\langle \cdot,\cdot\rangle_j$. 
First, we define the finite-dimensional tensor product of $H_j$, $j=1,...,m$, as the tensor vector space 
$H_1 \otimes H_2 \otimes \cdots \otimes H_m$ equipped with the inner product
\begin{equation} \label{tensorporoduct-innerporducts}
 \langle\otimes_{j=1}^m\phi_j,\otimes_{j=1}^m\psi_j\rangle 
:= \ 
\prod_{j=1}^m\langle\phi_j,\psi_j\rangle_j \, \quad 
\mbox{for all} \ \phi_j,\psi_j \in H_j. 
\end{equation}
By taking the completion under this inner product, the resulting Hilbert space is defined as the tensor product space $H_1 \otimes H_2 \otimes \cdots \otimes H_m$  of $H_j$, $j=1,...,m$.
Next, we consider the infinite-dimensional case.
If $H_j, j \in \NN$, is a collection of separable Hilbert spaces and $\xi_j, j \in \NN$, is a collection of unit vectors in these Hilbert spaces then the infinite tensor product $\otimes_{j \in \NN} H_j$  is the completion of the set of all finite linear combinations of simple tensor vectors $\otimes_{j \in \NN} \phi_j$ where all but finitely many of the $\phi_j$'s are equal to the corresponding $\xi_j$. 
The inner product of 
$\otimes_{j \in \NN} \phi_j$ and $\otimes_{j \in \NN} \psi_j$ is defined as in \eqref{tensorporoduct-innerporducts}.
For details on infinite tensor product of Hilbert spaces, see \cite{BR02}. 

Now, we will need a tensor product of Hilbert spaces of a special structure. 
Let $H_1$ and $H_2$ be two given infinite-dimensional separable Hilbert spaces. Consider the infinite tensor product Hilbert space
\begin{equation}\nonumber
 \Ll 
:= \ 
H_1^m \otimes H_2^\infty 
\quad \quad \mbox{ where } \quad \quad
H_1^m
:= \
\otimes_{j=1}^m H_1, \quad  
H_2^\infty
:= \ \otimes_{j=1}^\infty H_2. 
\end{equation}
In the following, we use the letters $I,J$ to denote either 
$\ZZ_+$ or $\ZZ$.  Recall also that we use the letter $\Ii$ to denote either $\ZZmp$ or $\ZZm$ and the letter $\Jj$ to denote either $\FF$  or $\ZZis$.
Let $\{\phi_{1,k}\}_{k \in I}$ and $\{\phi_{2,s}\}_{s \in J}$ be given orthonormal bases of $H_1$ and $H_2$, respectively. 
Then,  $\{\phi_{1,\bk}\}_{\bk \in \Ii}$ and $\{\phi_{2,\bs}\}_{\bs \in \Jj}$ are orthonormal bases of $H_1^m$ and $H_2^\infty$, respectively, 
where
\begin{equation} \nonumber
\phi_{1,\bk}
:= \ 
 \otimes_{j=1}^m \phi_{1,k_j}, \quad
\phi_{2,\bs} 
:= \ 
\otimes_{j=1}^\infty \phi_{2,s_j}.
\end{equation}
Moreover, the set $\{\phi_{\bk,\bs}\}_{(\bk,\bs) \in \Ii \times \Jj}$ is an orthonormal basis of $\Ll$, where
\begin{equation} \nonumber
 \phi_{\bk,\bs} 
:= \ 
 \phi_{1,\bk}  \otimes \phi_{2,\bs}.
\end{equation}
Thus, every $v \in \Ll$ can by represented by the series
\begin{equation} \nonumber
v
\ = \ 
\sum_{(\bk,\bs) \in \Ii \times \Jj} v_{\bk,\bs}\,  \phi_{\bk,\bs},
\end{equation}
where 
\begin{equation}\nonumber
v_{\bk,\bs}:=\langle v,\phi_{\bk,\bs} \rangle_{\Ll} =\left( \left(v,\phi_{1,\bk}\right)_{H_1^m},\phi_{2,\bs} \right)_{H_2^\infty}=\left( \left(v,\phi_{2,\bs}\right)_{H_2^\infty},\phi_{1,\bk} \right)_{H_1^m}
\end{equation}
is the $(\bk,\bs)$th coefficient of $v$ with respect to the orthonormal basis $\{\phi_{\bk,\bs}\}_{(\bk,\bs) \in \Ii \times \Jj}$. 
Furthermore, there holds Parseval's identity
\begin{equation} \nonumber
\|v\|_\Ll^2 
\ = \ 
\sum_{(\bk,\bs) \in \Ii \times \Jj} |v_{\bk,\bs}|^2.
\end{equation}

Now let us assume that a general sequence of scalars 
$\lambda := \{\lambda(\bk,\bs)\}_{(\bk,\bs) \in \Ii \times \Jj}$ with 
$\lambda(\bk,\bs) \not= 0$ is given.
Then, we define the associated space
\begin{equation}\label{Llambda}
\Ll^\lambda:=\left\{v \in \Ll \ : \ \text{there exists } \rep{v} \in \Ll \text{ such that } v := \sum_{(\bk,\bs) \in \Ii \times \Jj} \frac{\rep{v}_{\bk,\bs}}{\lambda(\bk,\bs)} \, \phi_{\bk,\bs}\right\}
\end{equation}
 
The norm of $\Ll^\lambda$ is defined by
\begin{equation}\label{P-Id[norm-Ll^lambda]}
\|v\|^{2}_{\Ll^\lambda} \ := \|\rep{v}\|^{2}_\Ll=\sum_{(\bk,\bs) \in \Ii \times \Jj} |\lambda(\bk,\bs)|^2 \,  |v_{\bk,\bs}|^2,
\end{equation}
where the last equality stems from Parseval's identity.

Define $\Jj_s:= \{\, \bs \in \Jj: \ \supp(\bs) \subset \{1, \cdots, s\}\, \}$.
We consider 
\begin{equation}\label{[Ll_s]}
\Ll_s
:= 
\left\{v 
 =  
\sum_{(\bk,\bs) \in \Ii \times \Jj_s} v_{\bk,\bs}\,  \phi_{\bk,\bs} \right\}
\quad \text{and} \quad 
\Ll^\lambda_s
:= \
\Ll^\lambda \cap \Ll_s.
\end{equation}
Next, let us assume that the general nonzero sequences of scalars 
$\lambda := \{\lambda(\bk,\bs)\}_{(\bk,\bs) \in \Ii \times \Jj}$ and 
$\nu := \{\nu(\bk,\bs)\}_{(\bk,\bs) \in \Ii \times \Jj}$ are given with associated spaces $\Ll^\lambda$ and $\Ll^\nu$ with corresponding norms and subspaces $\Ll^\lambda_s$ and $\Ll^\nu_\bs$, c.f. \eqref{[Ll_s]}.
As in Section \ref{sec:spatialreg}, we define for $T \ge 1$ the index-set 
\begin{equation} \nonumber
G_{\Ii \times \Jj}(T) 
:= \ 
\big\{(\bk,\bs)  \in \Ii \times \Jj: \, \frac{\lambda(\bk,\bs)}{\nu(\bk,\bs)}   \leq T\big\},
\end{equation}
which induces a subspace  
\begin{equation}\nonumber
\Pp(T):=\left\{g \in \Ll \ : \  v
\ = \ 
\sum_{(\bk,\bs) \in G_{\Ii \times \Jj}(T)}  \, v_{\bk,\bs}\,  \phi_{\bk,\bs}\right\} \subset \Ll.
\end{equation}
We are interested in the $\Ll^\nu$-norm approximation of elements from $\Ll^\lambda$  by elements from $\Pp(T)$. 
To this end, for $v \in \Ll$ and $T \ge 0$, we define the  operator $\Ss_T$ as
\begin{equation} \nonumber
\Ss_T(v)
:= \
\sum_{(\bk,\bs) \in G_{\Ii \times \Jj}(T)}  \, v_{\bk,\bs}\,  \phi_{\bk,\bs}.
\end{equation}
We make the assumption throughout this section that $G_{\Ii \times \Jj}(T)$ is a finite set for every $T \ge 1$.
Obviously, $\Ss_T$ is the orthogonal projection onto $\Pp(T)$. 
Furthermore, we define the set $G_{\Ii \times \Jj_s}(T)$, the subspace $\Pp_s(T)$
and the operator $\Ss_{s,T}(v)$ in the same way by replacing $\Jj$ by $\Jj_s$.

The following lemma gives an upper bound for the error of the orthogonal projection $\Ss_T$ with respect to the parameter $T$.
\begin{lemma} \nonumber
For arbitrary $T \geq 1$, we have
\begin{equation} 
\|v - \Ss_T(v)\|_{\nonumber\Ll^\nu}
\ \le  \
T^{-1}\|v\|_{\Ll^\lambda} \, , \qquad  \forall v \in \Ll^\lambda \cap  \Ll^\nu.
\end{equation}
\end{lemma}

Recall that $\Uu^\lambda$ is the unit ball in $\Ll^\lambda$, i.e., 
$\Uu^\lambda := \ \{v \in \Ll^\lambda: \|v\|_{\Ll^\lambda} \le 1\},
$
and denote by $\Uu^\lambda_s$ the unit ball in $\Ll^\lambda_s$, i.e., 
$
\Uu^\lambda_s := \ \{v \in \Ll^\lambda_s: \|v\|_{\Ll^\lambda_s} \le 1\}.
$ We then have the following corollary.
\begin{corollary} \label{corollary[|f - S_T(f)|]}
For arbitrary $T \geq 1$, 
\begin{equation} \nonumber
\sup_{v \in \Uu^\lambda} \ \inf_{w \in \Pp(T)} \|v - w\|_{\Ll^\nu}
\ =  \
\sup_{v \in \Uu^\lambda} \|v - \Ss_T(f)\|_{\Ll^\nu}
\ \le  \
T^{-1}. 
\end{equation}
\end{corollary}

Now we are in the position to give lower and upper bounds on the $\varepsilon$-dimension $n_\varepsilon(\Uu^\lambda, \Ll^\nu)$.
\begin{lemma} \label{lemma[n_e]}
Let $\varepsilon \in (0,1]$. Then, we have
\begin{equation} \nonumber
|G_{\Ii \times \Jj}(1/\varepsilon)| - 1
\ \le \
n_\varepsilon(\Uu^\lambda, \Ll^\nu)
\ \le  \
|G_{\Ii \times \Jj}(1/\varepsilon)|.
\end{equation}
\end{lemma}

In a similar way, by using
the set $G_{\Ii \times \Jj_s}(T)$, the subspace $\Pp_s(T)$
and the operator $\Ss_{d,T}(f)$, we can prove the following lemma for $n_\varepsilon(\Uu^\lambda_s, \Ll^\nu_\bs)$. 
\begin{lemma}  \label{lemma[n_e(s)]}
Let $\varepsilon \in (0,1]$. Then we have
\begin{equation} \nonumber
|G_{\Ii \times \Jj_s}(1/\varepsilon)| - 1
\ \le \
n_\varepsilon(\Uu^\lambda_s, \Ll^\nu_\bs)
\ \le  \
|G_{\Ii \times \Jj_s}(1/\varepsilon)|.
\end{equation}
\end{lemma}

These lemmas show that we need to estimate the cardinality of the index sets $|G_{\Ii \times \Jj}(1/\varepsilon)|$ and $|G_{\Ii \times \Jj_s}(1/\varepsilon)|$. We will treat this problem in Section \ref{cardinality of HC} for  infinite tensor product Hilbert spaces of joint regularity which are related to the solution of parametric PDEs.

\section{Joint regularity of the solution of parametric elliptic PDEs}
\label{solution joint regularity}
In order to apply our results on approximation in Section \ref{sec:infiniteapprox} to  the parametric elliptic model problem \eqref{SPDE} we show that the solution to this problem belongs to certain infinite tensor product Hilbert spaces of joint regularity. To this end, we combine the results from Subsections \ref{sec:spatialreg} and \ref{sec:paramreg} to derive explicit formulas for the sequences $\lambda$ and $\nu$ for these spaces.

We focus on functions defined in $L_2(\TTm) \otimes L_2(\IIi)$. 
Let $e_k(x):=  e^{i2\pi k x}$. 
Then $\{e_k\}_{k \in \ZZ}$ is an orthonormal basis of 
$L_2(\TT)$.
Let
$\{L_s\}_{s=0}^\infty$ be the family of univariate orthonormal Legendre polynomials in 
$L_2(\II)$. 
For $(\bk,\bs)\in \ZZm\times \FF$, we define
\begin{equation} \nonumber
L_{(\bk,\bs)}(\bx,\by)
:= \ 
e_\bk(\bx)L_\bs(\by), 
\quad e_\bk(\bx):= \ \prod_{j=1}^m e_{k_j}(x_j), \
L_\bs(\by):= \ \prod_{j \in \supp(\bs)} L_{s_j}(y_j).
\end{equation}
Note that $\{L_{(\bk,\bs)}\}_{(\bk,\bs) \in \ZZm\times \FF}$ is an orthonormal basis of 
$L_2(\TTm \times \IIi)$. Moreover,  we have the following expansion for every 
$v \in L_2(\TTm \times \IIi)$,
\begin{equation} \nonumber
v = \sum_{(\bk,\bs) \in \ZZm\times \FF} v_{\bk,\bs}L_{(\bk,\bs)},
\end{equation}
where for  $(\bk,\bs) \in \ZZm\times \FF$, 
$v_{\bk,\bs}: = \langle v, L_{(\bk,\bs)} \rangle$  denotes the $(\bk,\bs)$th Fourier coefficient of $v$ with respect to the orthonormal basis $\{L_{(\bk,\bs)}\}_{(\bk,\bs) \in \ZZm\times \FF}$.

We present two specific examples for sequences $\lambda$ and their associated function spaces $\Ll^{\lambda}$ which naturally arise in the regularity theory of parametric elliptic partial differential equations, in particular, 
 of problem \eqref{SPDE}.
Let the  pair $a,\bb$ be given by
 \begin{equation} \nonumber
\quad a >0; \quad 
\bb = (b_j)_{j \in \NN}, \  b_j > 0, \ j \in \NN.
\end{equation}
For each  $(\bk,\bs) \in \ZZms \times \FF$, we define the scalar $\rho(\bk,\bs)$ by
\begin{equation}\label{[lambda{a,mu,br}]}
\rho(\bk,\bs)
\ := \
\rho_{a,\bb}(\bk,\bs)
 := \
 \max_{1\le j \le m}|k_j|^a \, \frac{\bs!}{|\bs|_1!} \bb^{-\bs}.
\end{equation}

Then, we define the associated space 
\begin{eqnarray} \nonumber
&&\Ll^{\rho}=:A^{\alpha,\bb}(\TTm \times \IIi)\\&=&\left\{v \in L_2(\TTm \times \IIi) \ :  \ \text{there is } \rep{v}\in  L_2(\TTm \times \IIi) \text{ such that }v = \sum_{(\bk,\bs) \in \ZZms\times \FF} \frac{\rep{v}_{\bk,\bs}}{\rho_{\alpha,\bb}(\bk,\bs)} \, L_{(\bk,\bs)}  \right\} \nonumber .
\end{eqnarray}
The norm of $A^{\alpha,\bb}(\TTm \times \IIi)$ is defined by
\begin{equation}\nonumber
\|v\|^{2}_{A^{\alpha,\bb}(\TTm \times \IIi)} \ := \|\rep{v}\|^{2}_{L_2(\TTm \times \IIi)} = \sum_{(\bk,\bs) \in \ZZms \times \FF} \rho^{2}_{\alpha,\bb}(\bk,\bs) \left| v_{\bk,\bs}\right|^{2}.
\end{equation}

Next, we define 
\begin{equation}\nonumber
\theta(\bk,\bs)=\theta_{\beta}(\bk,\bs):=\left|\bk\right|^{\beta}_{\infty} 
\end{equation}
The Sobolev-type space 
\begin{eqnarray} \nonumber
&&\Ll^{\theta}=K^\beta(\TTm \times \IIi)\\&:=&\left\{v \in L_2(\TTm \times \IIi) \ : \ \text{there is } \rep{v} \in L_2(\TTm \times \IIi) \text{ such that } v := \sum_{(\bk,\bs) \in \ZZms \times \FF} \frac{\rep{v}_{\bk,\bs}}{\theta_{\beta}(\bk,\bs)} \, L_{(\bk,\bs)}\right\} \nonumber
\end{eqnarray}
Again, the norm of $K^\beta(\TTm \times \IIi)$ is defined by
\begin{equation}\nonumber
\|v\|^{2}_{K^\beta(\TTm \times \IIi)} \ := \|\rep{v}\|^{2}_{L_2(\TTm \times \IIi)}=\sum_{(\bk,\bs) \in \ZZms \times \FF} \left|\rep{v}_{\bk,\bs}\right|^{2}\theta^{2}_{\beta}(\bk,\bs).
\end{equation}

\begin{lemma} \label{lemma[|u|_LV<]}
We have 
\begin{equation} \nonumber
\|v\|_{L_2(\IIi,V,\mu)} 
\ \le \
2\pi \sqrt{m} \|v\|_{K^1(\TTm \times \IIi)}, \quad v \in K^1(\TTm \times \IIi).
\end{equation}
and 
\begin{equation} \nonumber
\|v\|_{L_2(\IIi,W,\mu)} 
\ \le \
(2\pi)^{2} m \|v\|_{K^2(\TTm \times \IIi)}, \quad v \in K^2(\TTm \times \IIi).
\end{equation}
\end{lemma}

\begin{proof}
For a function $v \in  K^1(\TTm \times \IIi)$ of the form
\begin{equation} \nonumber
v
\ = \
\sum_{(\bk,\bs) \in \ZZm\times \FF} v_{\bk,\bs} L_{(\bk,\bs)}
\ = \
\sum_{\bs \in \FF} v_\bs \, L_\bs,
\end{equation}
we have by \eqref{Parseval-V}
\begin{equation} \nonumber
\|v\|_{L_2(\IIi,V,\mu)}^2
\ = \
\sum_{\bs \in \FF} \|v_\bs\|_V^2
\ \le 
(2\pi)^{2} m\sum_{(\bk,\bs) \in \ZZms \times \FF} |\bk|_\infty^2  |v_{\bk,\bs}|^2
\ = \
(2 \pi)^{2} m \|v\|_{K^1(\TTm \times \IIi)}^2.
\end{equation}
Similarly, we obtain with \eqref{Parseval-W}
\begin{eqnarray} \nonumber
\|v\|_{L_2(\IIi,W,\mu)}^2
\ &=& \
\sum_{\bs \in \FF} \|v_\bs\|_W^2
\ \le 
(2\pi)^{4}m^{2}\sum_{(\bk,\bs) \in \ZZms \times \FF} |\bk|_\infty^4  |v_{\bk,\bs}|^2
=(2\pi)^{4}m^{2}\sum_{(\bk,\bs) \in \ZZms \times \FF} \theta_{2}
^{2}(\bk,\bs)  |v_{\bk,\bs}|^2
\ \\&=& \
(2\pi)^{4}m^{2} \|v\|_{K^2(\TTm \times \IIi)}^2. \nonumber
\end{eqnarray}
\hfill
\end{proof}

\begin{lemma} \label{lemma[b in ell_p]}
Let $0 < p \le \infty$ and $\bb=(b_j)_{j\in \NN1}$ be a positive sequence. Then the sequence 
$\big(\bb^\bs\big)_{\bs \in \FF}$ belongs to $\ell_p(\FF)$ 
if and only if $\|\bb\|_{\ell_\infty(\NN)} < 1$ and $\bb \in \ell_p(\NN)$.
\end{lemma}

\begin{proof}
The proof of this lemma is the same as that of Lemma 7.1 in \cite{CDS10b}.
\hfill
\end{proof}

\begin{lemma} 
Let the assumptions and notation of Lemma \ref{lemma|u_bs|_V} hold.
Let furthermore $\bc = (c_j)_{j \in \NN}$ be any positive sequence such that 
$c_j > 1$ and such that the sequence  $\bc^{-1} = (c_j^{-1})_{j \in \NN}$ belongs to $\ell_2(\NN)$. Then, for the sequence 
\[ 
\bb:= (b_j)_{j \in \NN}, \quad b_j:= c_j d_j,
\] 
 the solution $u$ to \eqref{SPDE} belongs to $A^{1,\bb}:= A^{1,\bb}(\TTm \times \IIi)$
and 
\begin{equation} \nonumber
\|u\|_{A^{1,\bb}} 
\ \le \
K\, \|\bc^{-1}\|_{\ell_2(\FF)}.
\end{equation} 
\end{lemma}
\begin{proof}
We have by equation \eqref{Parseval-V}, Lemma \ref{lemma[b in ell_p]} and Lemma \ref{lemma|u_bs|_V} 
\begin{equation} \nonumber
\begin{split}
\|u\|_{A^{1,\bb}}^2
\ &= \
\sum_{(\bk,\bs) \in \ZZms \times \FF} |\bk|_\infty^2 \left(\frac{\bs!}{|\bs|_1!}\bb^{-\bs}\right)^2 |u_{\bk,\bs}|^2
\ \le \
\sum_{\bs \in \FF} \left(\frac{\bs!}{|\bs|_1!}\bb^{-\bs}\right)^2 \|u_\bs\|_V^2
\\[1ex]
\ &\le \
K^2 \sum_{\bs \in \FF}\bc^{-2\bs}
\ < \ \infty. 
\end{split}
\end{equation}
\hfill
\end{proof}

In the same way, from Eq. \eqref{Parseval-W}, Lemma \ref{lemma[b in ell_p]} and Lemma \ref{lemma|u_bs|_W} we deduce the following result.
\begin{lemma} \label{lemma[u in A]}
Let the assumptions and notation of Lemma \ref{lemma|u_bs|_W} hold.
Let furthermore $\bc = (c_j)_{j \in \NN}$ be any positive sequence such that 
$c_j > 1$ and such that the sequence  $\bc^{-1} = (c_j^{-1})_{j \in \NN}$ belongs to $\ell_2(\NN)$. For the sequence 
\[ 
\bb:= (b_j)_{j \in \NN}, \quad b_j:= c_j d_j,
\] 
the solution $u$ to \eqref{SPDE} then belongs to $A^{2,\bb}:= A^{2,\bb}(\TTm \times \IIi)$
and 
\begin{equation} \label{[|u|_A<]}
\|u\|_{A^{2,\bb}} 
\ \le \
K\, \|\bc^{-1}\|_{\ell_2(\FF)}.
\end{equation} 
\end{lemma}

\section{The cardinality of  infinite-dimensional hyperbolic crosses}
\label{cardinality of HC}
For $T>0$, consider the hyperbolic cross 
\begin{equation}\nonumber
E_{a,\bb}(T) 
:= \ 
\big\{(\bk,\bs) \in \ZZms \times \FF: \rho_{a,\bb}(\bk,\bs) \leq T\big\},
\end{equation}
in the infinite-dimensional case, where we recall
\begin{equation}\nonumber
\rho_{a,\bb}(\bk,\bs)
 := \
 |\bk|_\infty^a \, \frac{\bs!}{|\bs|_1!} \bb^{-\bs}.
\end{equation}
In order to obtain estimates on the $\varepsilon$-dimension in the norm $K^\beta(\TTm \times \IIi)$ of the unit ball in $A^{\alpha,\bb}(\TTm \times \IIi)$, we want to employ Lemma \ref{lemma[n_e]} or Lemma 
\ref{lemma[n_e(s)]} respectively. This, however, needs an estimate on $n := |E_{a,\bb}(T)|$ with $a = \alpha - \beta$. In this section, we establish such an estimate for the cardinality of $E_{a,\bb}(T)$. 

As a preparatory step, we first have to study sharp conditions for the inclusion $\Big(\frac{|\bs|_1!}{\bs!}\bb^\bs\Big)_{\bs \in \FF} \in \ell_p(\FF)$ with $0 < p < \infty$.
The main difference to the existing literature is, that we explicitly allow for $p>1$.
This result, though it is of its own interest, will be used in defining the constant in \eqref{[constC]} for the cost estimate.

\subsection{A condition for summability of sequences} 
\label{condition for summability}
In this subsection, given a sequence $\bb=(b_j)_{j=1}^\infty$, 
we are interested in a necessary and sufficient condition for the inclusion
$\Big(\frac{|\bs|_1!}{\bs!}\bb^\bs\Big)_{\bs \in \FF} \in \ell_p(\FF)$ with $0 < p < \infty$. 
We first recall a previous result for the case $0 < p \le 1$ which has been proven in \cite{CDS10b}.

\begin{theorem} \label{CDS2010}
Let $0 < p \le 1$ and $\bb=(b_j)_{j=1}^\infty$ be a positive sequence. Then the sequence $\Big(\frac{|\bs|_1!}{\bs!}\bb^\bs\Big)_{\bs \in \FF}$ belongs to $\ell_p(\FF)$ 
if and only if $\|\bb\|_{\ell_1(\NN)} < 1$ and $\bb \in \ell_p(\NN)$.
\end{theorem}

As shown in \cite{CDS10b, CDS10a, {CCS}}, the $\ell_p(\FF)$-summability with some $0 < p < 1$ of the sequence of the energy norm of the coefficients in chaos polynomial Taylor and Legendre expansions, together with Stechkin's lemma plays a basic role  in construction of nonlinear $n$-term approximation methods for the solution of parametric and stochastic elliptic PDEs. 
The proof of this $\ell_p(\FF)$-summability relies upon Theorem \ref{CDS2010}.  

In the present paper, we need a necessary and sufficient condition on the sequence $\bb=(b_j)_{j=1}^\infty$ for 
the $\ell_p(\FF)$-summability of the sequence $\Big(\frac{|\bs|_1!}{\bs!}\bb^\bs\Big)_{\bs \in \FF}$ in the case 
$0 < p < \infty$ which is a basic condition for construction of a linear approximation by orthogonal projection in the space $K^\beta:= K^\beta(\TTm \times \TTi)$ for functions from 
$A^{\alpha,\bb}:= A^{\alpha,\bb}(\TTm \times \TTi)$ and hence, collective Galerkin approximation in the Bochner space 
$L_2(\IIi,V,\mu)$ of the solution $u$ of the parametric elliptic problem \eqref{SPDE}. This necessary and sufficient condition of the $\ell_p(\FF)$-summability in the case $1 < p < \infty$ as well as its proof are different from those in the case $0 < p \le 1$. In the proof, we use in particular,
the following well known inequality between the arithmetic and geometric means, see, e.g., 
\cite[2.5, pp. 17-18]{HLP34}. For nonnegative numbers $a_1,...,a_n$ and positive numbers $p_1,...,p_n$, there holds true the inequality
\begin{equation} \label{G<A}
a_1^{p_1}\cdots a_n^{p_n}
\ < \
\left(\frac{a_1p_1 + \cdots +a_n p_n}{p_1 + \cdots + p_n}\right)^{p_1 + \cdots + p_n}
\end{equation}  
unless all the $a_1,...,a_n$ are equal.

 \begin{theorem}\label{Thm[1<p<infty]}
Let $1< p < \infty$ and $\bb=(b_j)_{j=1}^\infty$ be a nonnegative sequence with infinitely many positive $b_j$. Then,   
the sequence $\left(\frac{|\bs|_1!}{\bs!} \bb^\bs  \right)_{\bs \in \FF} $ belongs to $\ell_p(\FF)$ if and only if   $\| \bb\|_{\ell_1(\mathbb{N})} \leq  1.$
\end{theorem}

\begin{proof} \\
{\em Necessity.} Assume that the sequence $\bb$ is given and $\| \bb\|_{\ell_1(\mathbb{N})} >  1$. 
Then we fix  a $J \in \NN$ large enough so that 
\[
B 
:= \ 
b_1+  \cdots + b_{J}
\ > \ 
1 .
\]
For each $s \in \mathbb{N}$,  we define $\bs^* \in \FF$ by 
\begin{equation} \nonumber
s^*_j= \left\lfloor s  \frac{b_j}{B} \right\rfloor +1 \ \text{ if } \ 1\leq j \leq J, 
\ \ \text{and } \ s^*_j=0 \ \text{ if } \ j > J.
\end{equation}
So $s^*_j \geq  s  \frac{b_j}{B}$ for every $1\leq j \leq J$, and then 
\[
\frac{|\bs^*|_1}{s^*_j} \geq \frac{s}{s^*_j} \geq \frac{s}{s  \frac{b_j}{B} +1}=\frac{B}{b_j} \left(\frac{1}{1+ \frac{B}{sb_j}} \right) 
\geq \frac{B}{b_j} \lambda_s,  \quad \forall 1\leq j\leq J,
\]
where
\[
\lambda_s= \min \left\{\left(1+ \frac{B}{sb_j}\right)^{-1}: \  1\leq j \leq J\right\}.
\]
Hence, we have
\begin{align*}
\frac{|\bs^*|_1!}{\bs^*!} \bb^{\bs^*}&= \frac{|\bs^*|_1!}{\sqrt{2\pi |\bs^*|_1}(|\bs^*|_1/e)^{|\bs^*|_1} } 
\left(\prod_{j=1}^{J}  \frac{\sqrt{2\pi s^*_j}(s^*_j/e)^{s^*_j} }{s^*_j!} \right) 
\left( \frac{\sqrt{2\pi |\bs^*|_1} }{\prod_{j=1}^{J} \sqrt{2\pi s^*_j}} \right) 
 \left( \prod_{j=1}^J \left( \frac{|\bs^*|_1}{s^*_j}b_j\right)^{s^*_j}\right)\\
& \geq  
(2\pi)^{(1-J)/2}\frac{|\bs^*|_1!}{\sqrt{2\pi |\bs^*|_1}(|\bs^*|_1/e)^{|\bs^*|_1} } 
\left(\prod_{j=1}^{J}  \frac{\sqrt{2\pi s^*_j}(s^*_j/e)^{s^*_j} }{s^*_j!} \right)
\left( \frac{|\bs^*|_1}{\prod_{j=1}^{J}  s^*_j} \right)^{1/2}  
(\lambda_s B )^{|\bs^*|_1}.
\end{align*}
Observe that there are  a number $\sigma > 1$ and a number ${\bar s}:= {\bar s}(J) \in \NN$ large enough such that 
\begin{equation} \nonumber
\lambda_s B  
\ \ge \
\sigma, \ \forall s \ge {\bar s}.  
\end{equation}
From the estimate
\begin{equation} \nonumber
\frac{|\bs^*|_1}{\prod_{j=1}^{J}  s^*_j}
\ \ge \
 J^J |\bs^*|_1^{1-J}  
\ \ge \
J^J (s + J)^{1-J}, 
\end{equation}
which stems from an application of \eqref{G<A} and the observation that $|\bs^*|_1 \le s+J$
and the Stirling formula
\[
\lim_{k \to \infty } \frac{k!}{\sqrt{2\pi k}\left(\frac{k}{e}\right)^k}=1,
\]
we obtain 
\begin{equation} \nonumber
\frac{|\bs^*|_1!}{\bs^*!} \bb^{\bs^*} 
\geq C_J(\lambda_s  B )^{|\bs^*|_1}(s + J)^{(1-J)/2}
\geq C_J \sigma^s(s + J)^{(1-J)/2},  \ \forall s \ge {\bar s}, 
\end{equation}
where $C_J$ is a positive constant depending on $J$ only.
Therefore, for arbitrary $s \ge {\bar s}$
\begin{equation} \nonumber
\sum_{\bs \in \FF} \left(\frac{|\bs|_1!}{\bs!} \bb^\bs  \right)^p 
\geq 
\left(\frac{|\bs^*|_1!}{\bs^*!} \bb^{\bs^*}  \right)^p 
\geq 
 C_{J}^{p}\left(\sigma^s (s + J)^{(1-J)/2}\right)^p
\ \to \
\infty, \ s \to \infty. 
\end{equation}
The necessity is proven.

\bigskip
\noindent
{\em Sufficiency.} Assume that the sequence $\bb$ is given and  $\|\bb\|_{\ell_1(\mathbb{N})} \leq 1.$ We fix an integer $m$ satisfying the inequality $m(p-1) > 2$. 
Since the sequence $\bb$ has infinitely many positive terms $b_j$, with out loss of generality we may assume that $b_j > 0$ for all $j = 1,...,m+1$. 
Put $\bs = (\bs', \bs^{''})$ with $\bs' = (s_1,...,s_{m})$ and $\bs^{''} = (s_{m+1},s_{m+2},...)$ for
$\bs \in \FF$.
 We have 
\begin{align} \nonumber
\notag \sum_{\bs \in \FF} \left(\frac{|\bs|_1!}{\bs!} \bb^\bs  \right)^p 
&=
\sum_{M=0}^\infty \sum_{\bs \in \FF, |\bs|_1=M} \left(\frac{|\bs|_1!}{\bs!} \bb^\bs  \right)^p \\ \nonumber
&= 
\sum_{M=0}^\infty \sum_{\bs': |\bs'|_1 \leq  M}  
\sum_{\bs^{''} \in \FF: |\bs^{''}|_1=M-|\bs'|_1 } 
\left( \frac{|\bs|_1!}{\bs!} \bb^\bs  \right)^p \\  \nonumber
& \leq 
\sum_{M=0}^\infty \sum_{\bs': |\bs'|_1 \leq  M}    
\left(\sum_{\bs^{''} \in \FF, |\bs^{''}|_1=M-|\bs'|_1} \frac{|\bs|_1!}{\bs!} \bb^\bs   \right)^p .\nonumber 
\end{align}
Note that  
\[
\sum_{\bs^{''} \in \FF, |\bs^{''}|_1=M-|\bs'|_1} 
\frac{|\bs|_1!}{\bs!} \bb^\bs  =  \frac{M!}{\bk!} {\ba^\bk}, 
\]
where for convenience we redefined
\[
\bk = (k_1, k_2,\dots, k_m, k_{m+1})= (s_1, s_2,\dots, s_m, M-|\bs'|_1),
\]
\[
\ba = (a_1, a_2,\dots, a_m, a_{m+1})= (b_1, b_2, \dots,b_m, b_{m+1}+b_{m+2} + \dots).
\]
Hence, 
\begin{align} \nonumber
\sum_{\bs \in \FF} \left(\frac{|\bs|_1!}{\bs!} \bb^\bs  \right)^p &\leq 
\sum_{M=0}^\infty \sum_{\bk\in \mathbb{Z}_{+}^{m+1}: |\bk|_1 =  M} 
\left( \frac{M!}{\bk!} {\ba^\bk} \right)^p.
\end{align}
Putting 
\[
J_{1,M}:= \{\bk\in \mathbb{N}^{m+1}: |\bk|_1 =  M \}, \quad 
J_{2,M}:= \{\bk\in \mathbb{Z}_+^{m+1}: |\bk|_1 =  M, \prod_{j=1}^{m+1} k_j=0\},
\]
we obtain 
\begin{equation} \nonumber
\sum_{\bs \in \FF} \left(\frac{|\bs|_1!}{\bs!} \bb^\bs  \right)^p 
\leq 
\sum_{M=0}^\infty \sum_{\bk\in J_{1,M}} \left( \frac{M!}{\bk!} {\ba^\bk} \right)^p 
 + \sum_{M=0}^\infty\sum_{\bk\in J_{2,M}}  \left( \frac{M!}{\bk!} {\ba^\bk} \right)^p 
=: I_1 + I_2 .
\end{equation}
We have
\begin{align*}
&I_2 \leq \\&\sum_{M=0}^\infty \sum_{j=1}^{m+1} \sum_{\bk: |(k_1, \dots,k_{j-1}, k_{j+1}, \dots,k_{m+1})|_1 =  M} 
 \left( \frac{M!}{(k_1, \dots,k_{j-1}, k_{j+1}, \dots,  k_{m+1})!} a_1^{k_1} \dots a_{j-1}^{k_{j-1}} a_{j+1}^{k_{j+1}} \dots   a_{m+1}^{k_{m+1}} \right)^p.
\end{align*}
On the other hand,
\begin{align*}
&\sum_{j=1}^{m+1} \sum_{\bk: |(k_1, \dots,k_{j-1}, k_{j+1}, \dots,k_{m+1})|_1 =  M} 
 \left( \frac{M!}{(k_1, \dots,k_{j-1}, k_{j+1}, \dots,  k_{m+1})!} a_1^{k_1} \dots a_{j-1}^{k_{j-1}} a_{j+1}^{k_{j+1}} \dots   a_{m+1}^{k_{m+1}} \right)^p \\
&\leq \sum_{j=1}^{m+1} \left( \sum_{\bk: |(k_1, \dots,k_{j-1}, k_{j+1}, \dots,k_{m+1})|_1 =  M} 
  \frac{M!}{(k_1, \dots,k_{j-1}, k_{j+1}, \dots,  k_{m+1})!} a_1^{k_1} \dots a_{j-1}^{k_{j-1}} a_{j+1}^{k_{j+1}} \dots a_{m+1}^{k_{m+1}} \right)^p\\
&= \sum_{j=1}^{m+1} \Big(a_1 + \dots a_{j-1} + a_{j+1} + \dots  + a_{m+1} \Big)^{Mp}
=: \sum_{j=1}^{m+1} A_j^{Mp}.
\end{align*}
Since $\bb$ is a nonnegative sequence with infinitely many positive terms $b_j$, and $\|\bb\|_{\ell_1(\NN)} \le 1$, we deduce that $\ba$ is a positive vector in $\RR^{m+1}$ with $a_1+\cdots+a_{m+1} \le 1$, and 
consequently, $A_j < 1$ for $1 \le j \le m+1$. Hence,
\[
I_2 \leq  \sum_{j=1}^{m+1} \sum_{M=0}^\infty A_j^{Mp} < \infty.
\]

Let us estimate $I_1$. Putting 
\[
J_{3,M}= \{\bk\in \mathbb{N}^{m+1}: |\bk|_1 =  M, \frac{a_jk_i}{a_i k_j} \in [1/2, 2] \ \text{ for all } \ 
i, j =1,..., m + 1 \},
\]
\[
J_{4,M}= \{\bk\in \mathbb{N}^{n+1}: |\bk|_1 =  M, \frac{a_jk_i}{a_i k_j} \not\in [1/2, 2] \ \text{ for some  } \ i, j =1,..., m + 1\},
\]
we split $I_1$ into two sums $I_3$ and $I_4$ as 
\begin{equation} \nonumber
I_1
\ = \
\sum_{M=0}^\infty \sum_{\bk\in J_{3,M}} \left( \frac{M!}{\bk!} {\ba^\bk} \right)^p 
 + \sum_{M=0}^\infty\sum_{\bk\in J_{4,M}}  \left( \frac{M!}{\bk!} {\ba^\bk} \right)^p 
=: I_3 + I_4.
\end{equation}
By Stirling's approximation,
\begin{equation} \nonumber
\frac{M!}{\bk!} {\ba^\bk}
\ \le  \
{C}\,(2\pi)^{-m/2} \prod_{j=1}^{m+1} (M a_j/k_j)^{k_j}  \left(M\prod_{j=1}^{m+1} k_j^{-1}\right)^{1/2},  
\end{equation}
where $C$ is an absolute constant.

We estimate $I_3$. For all $\bk \in J_{3,M}$, we have by definition 
\[
a_j/ k_j \leq 2 (a_1 + ... + a_{m+1})/ (k_1 + ... + k_{m+1})\leq 2/M,
\]
and therefore, 
\[
k_j^{-1} \leq  2/(a_{\min} M),
\]
where 
\[
a_{\min} = \min \{ a_1, \cdots , a_{m+1}\} >0.
\]
Also, as mentioned above, we have $a_1 +  \cdots  + a_{m+1} \le 1$.
All these together with the inequality \eqref{G<A} give
\begin{equation}\nonumber
\prod_{j=1}^{m+1} (M a_j/k_j)^{k_j}  \left(M\prod_{j=1}^{m+1} k_j^{-1}\right)^{1/2} 
\leq \left(\sum_{j=1}^{m+1} a_j\right)^M  2^{m+1}\left(a_{\min}^{-m - 1}M^{-m}\right)^{1/2}
\leq 2^{m+1}a_{\min}^{-(m+1)/2}M^{-m/2}
\end{equation}
Therefore,
\begin{equation} \label{estimation00}
\begin{aligned}
\sum_{\bk\in J_{3,M}} \left( \frac{M!}{\bk!} {\ba^\bk} \right)^p 
&\leq  
\left( \frac{M!}{\bk!} {\ba^\bk} \right)^{p-1}
\sum_{\bk \in J_{3,M}} \frac{M!}{\bk!} {\ba^\bk}
\\
&\leq 
C_3 M^{-m(p-1)/2}
\sum_{\bk \in J_{1,M}} \frac{M!}{\bk!} {\ba^\bk}
\\
&\leq 
C_3 M^{-m(p-1)/2}
\left(\sum_{j=1}^{m+1} a_j\right)^M
\leq C_3 M^{-m(p-1)/2}.
\end{aligned}
\end{equation}
This and the inequality $m(p-1)/2 > 1$ imply that 
\[
I_3 \leq C_3\sum_{M=0}^\infty M^{-m(p-1)/2}
\ < \
\infty.
\]

Now, we estimate $I_4$. Take any  $ \bk\in J_{4,N}$, and rearrange 
$(1,2, \dots, m+1)$ to $(i_1, i_2, \dots, i_{m+1})$ so that
\begin{align}\label{dd1}
\frac{a_{i_1}}{{k_{i_1}}} \geq \frac{a_{i_2}}{k_{i_2}} \geq \cdots \geq \frac{a_{i_{m+1}}}{k_{i_{m+1}}}.
\end{align}
Then denoting $\alpha_{i_j}:= \frac{a_{i_j}}{{k_{i_j}}}$, by definition we have
\[
\frac{\alpha_{i_1}}{\alpha_{i_{m+1}}} 
> 2,
\]
Therefore, since
\[
\frac{\alpha_{i_1}}{\alpha_{i_2}}\cdot\frac{\alpha_{i_2}}{\alpha_{i_3}} \cdots \frac{\alpha_{i_m}}{\alpha_{i_{m+1}}}
= \frac{\alpha_{i_1}}{\alpha_{i_{m+1}}} 
> 2,
\]
there exists $\nu \in \mathbb{N}, 1\leq \nu \leq m+1$ such that
\begin{equation} \label{ineq:>sqrt[m]{2}}
\frac{\alpha_{i_\nu}}{\alpha_{i_{\nu+1}}} \geq \sqrt[m]{2}.
\end{equation}
From (\ref{dd1}) we have 
\[
\frac{a_{i_1} + a_{i_2}  + \dots + a_{i_\nu}}{k_{i_1} + k_{i_2}  + \cdots 
+ k_{i_\nu}}\geq \frac{a_{i_\nu}}{{k_{i_\nu}}}
= \alpha_{i_\nu}
\]
and
\[
\frac{a_{i_{\nu+1}} + a_{i_{\nu+2}} + \cdots + a_{i_{m+1}}}{k_{i_{\nu+1}} + k_{i_{\nu+2}}  + \cdots + k_{i_{m+1}}}
\leq \frac{a_{i_{\nu+1}}}{{k_{i_{\nu+1}}}}
= \alpha_{i_{\nu+1}}.
\]
We define   the nonempty sets: $e=\{i_1, i_2, \dots i_\nu\} \subset \{1,2,..,m\}$ and $e'=\{1,2,..,m\} \setminus e$. From  \eqref{ineq:>sqrt[m]{2}} we obtain
\begin{equation} \label{>sqrt[m]{2}}
\frac{ (\sum_{j \in e} a_j) / (\sum_{j \in e} k_j) }{ (\sum_{j \in e'} a_j) / (\sum_{j \in e'} k_j ) } 
\ \geq \ 
\sqrt[m]{2}.
\end{equation}
Therefore, by the inequality \eqref{G<A},
\begin{equation} \label{estimation01}
\begin{aligned}
& \prod_{j=1}^{m+1} (M a_j/k_j)^{k_j}  \left(M\prod_{j=1}^{m+1} k_j^{-1}\right)^{1/2}  \\
& \leq  M^{1/2}\prod_{j=1}^{m+1} (M a_j/k_j)^{k_j} \\
&\leq    M^{1/2} \left( \frac{M \sum_{j \in e} a_j }{ \sum_{j \in e} k_j} \right)^{\sum_{j \in e} k_j}  
\left( \frac{M \sum_{j \in e'} a_j }{ \sum_{j \in e'} k_j} \right)^{\sum_{j  \in e'} k_j} \\
&=:   M^{1/2}  \left( \frac{M c_1 }{r_1} \right)^{r_1}  \left( \frac{M c_2}{r_2} \right)^{r_2},  
\end{aligned}
\end{equation}
with $r_1 +r_2 =M$ and $c_1 + c_2 = \|\bb\|_{\ell_1(\NN)} \le 1$, where 
\[
c_1= \sum_{j \in e} a_j, c_2 =\sum_{j \in e'} a_j, r_1 =\sum_{j \in e} k_j, r_2= \sum_{j \in e'} a_j. 
\]

Consider the function 
\[
h(x)=\left(\frac{x}{M c_1} \right)^x \left(\frac{M-x}{M c_2} \right)^{M-x},  \quad x \in (0,M).
\]
Notice that it has an absolute minimum in the interval $(0,M)$ at the point $x_{\min} = \frac{Mc_1}{c_1 + c_2}$, and is decreasing in the interval $(0,x_{\min})$ and increasing in the interval $(x_{\min},M)$. 
By \eqref{>sqrt[m]{2}} we have 
\[
\frac{c_1/r_1}{c_2/(M-r_1)} \geq  \sqrt[m]{2}
\]  
which implies that
\[
0
\ < \
r_1 
\  \le \ 
\frac{Mc_1}{c_1 +  c_2\sqrt[m]{2}} 
\ < \ \frac{Mc_1}{c_1 + c_2} \ = \ x_{\min}, 
\]
and therefore, 
\begin{align*}
(M c_1/r_1)^{r_1} (M c_2/(M- r_1))^{(M- r_1)}= 1/f(r_1)
\leq 1/f(Mc_1/ (c_1  + c_2 \sqrt[m]{2}))
\ = \ \delta^M, 
\end{align*}
where 
\begin{equation} \label{delta}
\delta:= \
 \left(c_1  + c_2\sqrt[m]{2}\right)^{c_1/(c_1  + c_2\sqrt[m]{2} )}  
\left( \frac{c_1  + c_2\sqrt[m]{2}}{\sqrt[m]{2}}\right)^{(c_2\sqrt[m]{2})/(c_1  + c_2\sqrt[m]{2} )}. 
\end{equation}
Combining this with  \eqref{estimation01} we obtain
\begin{align} \nonumber
 \prod_{j=1}^{m+1} (M a_j/k_j)^{k_j}  \left(M\prod_{j=1}^{m+1} k_j^{-1}\right)^{1/2}  
\ \le \
\delta^M  M^{1/2}.   
\end{align}
Hence, similarly to \eqref{estimation00} we derive that
\begin{align} \label{ineq01}
\sum_{\bk\in J_{4,M}} \left( \frac{M!}{\bk!} {\ba^\bk} \right)^p 
\ \leq  \
C_4 \delta^{(p-1)M}  M^{(p-1)/2}.
\end{align}
Observe that by the construction for the given sequence $\bb$ and number $m$, the positive numbers $c_1, c_2$ and therefore,  the positive number $\delta$ as defined in \eqref{delta} depend only on the nonempty set 
$e \subset \{1,...,m\}$, i.e., $c_1 = c_1(e)$, $c_2= c_2(e)$ and $\delta = \delta(e)$. Consider the production in 
the right hand of \eqref{delta}. Since
\[
c_1(e)  + c_2(e)\sqrt[m]{2}
\ > \  
\frac{c_1(e)  + c_2(e)\sqrt[m]{2}}{\sqrt[m]{2}},
\]
applying the inequality \eqref{G<A} to this production with $c_1(e_{\max}), c_2(e_{\max})$, gives 
for all the nonempty sets $e \subset \{1,...,m\}$,
\[
0 
\ < \
\delta(e)
\ \le \
\delta_{\max}
:= \ 
\delta(e_{\max})
\ < \
c_1(e_{\max}) + c_2(e_{\max})
\ \le \
1,
\]
where $e_{\max} \subset \{1,...,m\}$ is a set such that
\[
\delta(e_{\max})
\ = \
\max_{e \subset \{1,...,m\}, \ e \not= \varnothing} \delta(e).
\]
Thus, provided with \eqref{ineq01} and $\delta \le \delta_{\max} < 1$, we arrive at
\[
I_4 \leq  C_4 \sum_{M=0}^\infty\delta_{\max}^{(p-1)M}  M^{(p-1)/2}
\ < \ \infty.
\]
The proof of sufficiency is complete.
\hfill
\end{proof}

In Theorem \ref{Thm[1<p<infty]}, the assumption that the nonnegative sequence $\bb=(b_j)_{j=1}^\infty$ has infinitely many positive $b_j$, is essential. Indeed,  if $\bb= (b_1,b_2,0,0,\dots)$ with  $b_1=b_2=1/2$, then a computation shows that
$\left( \frac{|\bs|_1!}{\bs!}\bb^s  \right)_{s \in \FF} \not\in l_p(\FF)$ for all $p\leq 2$. However, one can prove that for $3< p < \infty$  and  any non-negative sequence $\bb=(b_j)_{j=1}^\infty$,  
the sequence $\left(\frac{|\bs|_1!}{\bs!} \bb^s  \right)_{s \in \FF} $ belongs to $l_p(\FF)$ if and only if   $\|\bb\|_{l_1(\mathbb{N})} \leq   1$.
For application we will consider only positive sequences $\bb=(b_j)_{j=1}^\infty$ when this assumption always holds.

\subsection{Estimates of the cardinality of  infinite-dimensional hyperbolic crosses}
\label{Estimates of cardinality of HC}

\bigskip
We are now in the position to derive an estimate for the cardinality of $E_{a,\bb}(T)$.
\bigskip
\begin{theorem} \label{theorem[<E<]}
Let  $a > 0$, $\bb = (b_j)_{j \in \NN}$  be a positive sequence. Then
\begin{equation}\label{bb<}
|E_{a,\bb}(T)| < \infty, \ \forall T \ge 1, \ \Longleftrightarrow \
\begin{cases} 
\|\bb\|_{\ell_1(\NN)} \ < \ 1, \quad \bb \in \ell_{m/a}(\NN), \ & \ m/a \le 1, 
\\[1ex]
\|\bb\|_{\ell_1(\NN)} \ \le \ 1,  \ & \ m/a > 1.
\end{cases}
\end{equation}
Under this assumption, we have for every $T \ge 1$,
\begin{equation} \label{[<E<]}
2^m \left(\lfloor T^{1/a} \rfloor - 1\right)^m
\ \le \ 
|E_{a,\bb}(T)|
\ \le \
2^m\,C\, T^{m/a},
\end{equation}
where
\begin{equation} \label{[constC]}
C:= \
\left(\frac 3 2 \right)^{2m} \sum_{\bs \in \FF} \biggl(\frac{|\bs|_1!}{\bs!}\bb^\bs\biggl)^{m/a}.
\end{equation}
\end{theorem}

\medskip
\begin{proof} We first prove the sufficiency of \eqref{bb<} and \eqref{[<E<]} together, 
and then the necessity of \eqref{bb<}. Assume that there holds the condition on the sequence $\bb$ in the right hand side of \eqref{bb<}.
Let $T \ge 1$ be given. Observe that 
$
|E_{a,\bb}(T)| \ = \ 2^m|E^*(T)|,
$
where
\begin{equation} \nonumber
E^*(T) 
:= \ 
\big\{(\bk,\bs) \in \NN^m \times \FF: \rho(\bk,\bs) \leq T\big\}.
\end{equation}
Thus, we need to derive an estimate for $|E^*(T)|$. 
To this end, for $\bs \in \FF$, we put  
\[
T_{\bs} \ := \ T^{1/a} \, \biggl(\frac{|\bs|_1!}{\bs!} \bb^\bs \biggl)^{1/a}.
\]
By definition and the symmetry of the variable $k_j$ we have
\begin{equation}\nonumber
\begin{split}
|E^*(T)| 
\ &= \
\sum_{\bs \in \FF} \quad 
\sum_{\bk \in \NN^m: \ 
|\bk|_\infty \ \le \ T_{\bs}} 1 
\  \le \
m \sum_{\bs \in \FF} \quad 
\sum_{\substack{\bk \in \NN^m: \ 
k_m \ \le \ T_{\bs} \\ k_j \le k_m, \ j = 1,...,m-1}} 1 \\[2ex]
\ & \le \
m \sum_{\bs \in \FF} \quad 
\sum_{k \in \NN: \ k \ \le \ T_{\bs}} k^{m-1}.
\end{split}
\end{equation}
Hence,  since $T_{\bs} \ge 1$, applying Lemma 2.3 in \cite{DG16} gives
\begin{equation} \nonumber
\begin{split}
|E^*(T)| 
\ &\le \
 m \sum_{\bs \in \FF} \frac{1}{m} \left(\frac{3}{2}\right)^m (T_{\bs} + 1/2)^m  
\ \le \
 m \sum_{\bs \in \FF} \frac{1}{m} \left(\frac{3}{2}\right)^m (T_{\bs} + T_{\bs}/2)^m  \\[2ex]
\ &\le \
\left(\frac{3}{2}\right)^m  \, \sum_{\bs \in \FF} \left(\frac{3}{2}\right)^m \,T_{\bs}^m  
\ \le \
\left(\frac{3}{2}\right)^{2m} \, \sum_{\bs \in \FF} T_{\bs}^m  \\[2ex]
\ &= \
\left(\frac{3}{2}\right)^{2m} \, T^{m/a}  \,  
\sum_{\bs \in \FF} \biggl(\frac{|\bs|_1!}{\bs!}\bb^\bs \biggl)^{m/a}  
\\[2ex]
\ &\le \
\left(\frac{3}{2}\right)^{2m} \, T^{m/a}  \,  
\sum_{\bs \in \FF} \biggl(\frac{|\bs|_1!}{\bs!} \bb^\bs \biggl)^{m/a}.
\end{split}
\end{equation}
Due to the assumption of theorem, by Theorems \ref{CDS2010} and \ref{Thm[1<p<infty]} the sum in the right hand of the last inequality is finite. Thus, the upper bound in \eqref{[<E<]} is proven. 
The lower bound can be proven in the same way as that for \cite[Theorem 2.13]{DG16}.

To complete the proof we verify the necessity of  \eqref{bb<}.
Indeed, we have
\begin{equation} \nonumber
\begin{split}
|E^*(T)| 
\ &= \
\sum_{\bs \in \FF} \quad 
\sum_{\bk \in \NN^m: \ 
\max_{1\le j \le m} k_j \ \le \ T_{\bs}} 1 
\  \ge \
\sum_{\bs \in \FF} \quad 
\sum_{\substack{\bk \in \NN^m: \ 
k_m \ \le \ T_{\bs} \\ k_j \le k_m, \ j = 1,...,m-1}} 1 
\\[1.5ex]
\ & \ge \
C_1 \sum_{\bs \in \FF} \quad 
\sum_{k \in \NN: \ k \ \le \ T_{\bs}} k^{m-1} 
\\[1.5ex]
\ &\ge \
C_2 \sum_{\bs \in \FF} T_{\bs}^m  
\\[1.5ex]
\ &= \
C_2 \, T^{m/a}  \,  
\sum_{\bs \in \FF} \biggl(\frac{|\bs|_1!}{\bs!}\bb^\bs \biggl)^{m/a}.  
\end{split}
\end{equation}
We know from Theorems \ref{CDS2010} and \ref{Thm[1<p<infty]} that the last sum over $\bs \in \FF$ is finite only if there holds the condition on the sequence $\bb$ in the right hand side of \eqref{bb<}.
This proves the necessary.
\hfill
\end{proof}



\section{Final approximation rates}
\label{finalrates}
\subsection{{$\varepsilon$-dimension and $n$-widths}}
\label{e-dimnsion and n-widths}
For a finite subset $G$ in $ \ZZms \times \FF$, 
denote by $\Vv(G)$ the subspace in $L_2(\TTm \times\IIi)$ of  all functions $f$
of the form
\begin{equation} \nonumber
v
\ = \
\sum_{(\bk,\bs) \in G} v_{\bk,\bs}L_{(\bk,\bs)} 
\end{equation}
and define the linear operator $\Ss_G: \, L_2(\TTm \times\IIi) \to \Vv(G)$ by
\[
\Ss_G v
:= \
\sum_{(\bk,\bs) \in G} v_{\bk,\bs}L_{(\bk,\bs)}.
\]
Moreover, let $\Ss_{s,G}$ be the restriction of the operator $\Ss_G$ on $L_2(\TTm \times \II^s)$.

Then, for $s \in \NN$, we define the spaces $A_s^{\alpha,\bb}(\TTm \times \IIi)$,  $K_s^\beta(\TTm \times \IIi)$ and 
$\Vv_s(G)$  as the intersections of
$A^{\alpha,\bb}(\TTm \times \IIi)$,  $K^\beta(\TTm \times \IIi)$ and $\Vv(G)$ with $L_2(\TTm \times \II^s)$. 
Furthermore, let $U^{\alpha,\bb}(\TTm \times \IIi)$ and $U_s^{\alpha,\bb}(\TTm \times \IIi)$ be the unit ball in $A^{\alpha,\bb}(\TTm \times \IIi)$ and $A_s^{\alpha,\bb}(\TTm \times \IIi)$, respectively. 
In the following theorems, we drop for convenience $(\TTm \times \IIi)$ from the relevant notations. For example, we write  
$U^{\alpha,\bb}$ instead $U^{\alpha,\bb}(\TTm \times \IIi)$.

From the results on the cardinality of  infinite-dimensional hyperbolic crosses in Section 
\ref{cardinality of HC} and the results on approximation in infinite tensor product Hilbert spaces in 
Section~\ref{sec:infiniteapprox} we can now deduce results on approximation in the norm of $K^{\beta}$ 
of functions from $U^{\alpha,\bb}$ and in the norm of $K^{\beta}_s$ 
of functions from $U^{\alpha,\bb}_s$ in terms of $\varepsilon$-dimension and $n$-widths as follows.

\begin{theorem} \label{theorem[n_e]-nonperiodic}
Let $\alpha > \beta \ge 0$ and $\bb = (b_j)_{j \in \NN }$ be a positive sequence.
Suppose  that there hold the assumptions of Theorem \ref{theorem[<E<]} for $a = \alpha - \beta > 0$ and the sequence  
$\bb$. We have for every $s \in \NN$ and every $\varepsilon \in (0,1]$,
\begin{equation} \label{ineq[n_e<]-nonperiodic}
2^m \left(\lfloor \varepsilon^{- 1/(\alpha - \beta)} \rfloor - 1\right)^m
\ \le \
n_\varepsilon {(U_s^{\alpha,\bb}, K^{\beta}_s)}
\ \le \
n_\varepsilon(U^{\alpha,\bb}, K^{\beta})
\ \le \
2^m C \, \varepsilon^{- m/(\alpha - \beta)}, 
\end{equation}
where $C$ is the constant defined in \eqref{[constC]}. 
\end{theorem}

\begin{proof}  By putting $\Ii :=  \ZZm$ and $\Jj:= \FF$; $H_1 =L_2(\TT)$ and $H_2 = L_2(\II)$; 
$\phi_{1,k} : = e_k$ and $\phi_{2,s} : = L_s$;
$\lambda(\bk,\bs) := \rho_{\alpha,\bb}(\bk,\bs)$; $\nu(\bk,\bs) := |\bk|_\infty^\beta$, we have 
$\Ll = L_2(\TTm \times \IIi)$; $\Uu^\lambda  = U^{\alpha,\bb}$; $\Ll^\nu = K^{\beta}$. Then the inequalities in 
\eqref{ineq[n_e<]-nonperiodic} follow from Lemmas \ref{lemma[n_e]} and \ref{lemma[n_e(s)]} and 
Theorem~\ref{theorem[<E<]}.
\hfill
\end{proof}

Similarly, from Corollary~\ref{corollary[|f - S_T(f)|]} and Theorem~\ref{theorem[<E<]} we obtain

\begin{theorem} \label{theorem[d_n]-nonperiodic}
Under the assumptions of Theorem \ref{theorem[n_e]-nonperiodic},
with $E(T):=E_{\alpha-\beta,\bb}(T)$ and $n := |E(T)|$ we have
\begin{equation} \nonumber
\begin{split}
d_n(U^{\alpha,\bb}, K^{\beta})
\ &\le \
\sup_{v \in U^{\alpha,\bb}} \inf_{g \in \Vv(E(T))}\|v - g\|_{K^{\beta}}
\\[1ex]
\ &= \
\sup_{v \in U^{\alpha,\bb}} \|v - \Ss_{E(T)}(v)\|_{K^{\beta}}
\ \le \
2^{\alpha - \beta} C^{(\alpha - \beta)/m} \, n^{- (\alpha - \beta)/m}, 
\end{split}
\end{equation}
and for every $s \in \NN$,
\begin{equation} \nonumber
\begin{split}
d_n(U_s^{\alpha,\bb}, K^{\beta}_s)
\ &\le \
\sup_{v \in U_s^{\alpha,\bb}} \inf_{g \in \Vv_s(E(T))}\|v - g\|_{K^{\beta}_s}
\\[1ex]
\ &= \
\sup_{v \in U_s^{\alpha,\bb}} \|v - \Ss_{s,E(T)}(v)\|_{K^{\beta}_s}
\ \le \
2^{\alpha - \beta} C^{(\alpha - \beta)/m} \, n^{- (\alpha - \beta)/m}, 
\end{split}
\end{equation}
where $C$ is the constant defined in \eqref{[constC]}. 
\end{theorem}

Notice that from Theorem~\ref{theorem[n_e]-nonperiodic} one can also derive the lower bound
\begin{equation} \nonumber
d_n(U^{\alpha,\bb}, K^{\beta})
\ \ge \
d_n(U_s^{\alpha,\bb}, K^{\beta}_s)
\ \ge \
C' \, n^{- (\alpha - \beta)/m}, 
\end{equation}
where $C'$ is a positive constant depending on $\alpha, \beta, m$ only.

\subsection{Application to Galerkin approximation of parametric elliptic PDEs}
We now apply our results on the $\varepsilon$-dimension and $n$-widths of Subsection \ref{e-dimnsion and n-widths}
to the Galerkin approximation of parametric elliptic PDEs \eqref{SPDE}.

Since $u \in L_2(\IIi,V,\mu)$, it can be defined as the unique solution of the variational problem: Find  $u \in L_2(\IIi,V,\mu)$ such that
\begin{equation} \nonumber
B(u,v)
\ = \
F(v) \quad \forall v \in  L_2(\IIi,V,\mu),
\end{equation}
where
\begin{equation} \nonumber
\begin{split}
B(u,v)
&:= \
\int_{\IIi}\int_{\TTm} a(\bx,\by)\nabla u(\bx,\by) \cdot \nabla v(\bx,\by) \, \mbox{d}\bx\, \mbox{d}\mu(\by),
\\[1ex]
F(v)
&:= \
\int_{\IIi}\int_{\TTm} f(\bx) \, v(\bx,\by) \, \mbox{d}\bx\, \mbox{d}\mu(\by).
\end{split}
\end{equation}

We define the {\em Galerkin approximation} $u_G$ to $u$ as the unique solution to the problem:
Find  $u_G \in \Vv(G)$ such that
\begin{equation} \nonumber
B(u_G,v)
\ = \
F(v) \quad \forall v \in  \Vv(G).
\end{equation}
By C\'ea's lemma we have the estimate 
\begin{equation} \nonumber
\|u - u_G\|_{L_2(\IIi,V,\mu)}
\ \le \
\sqrt{\frac{R}{r}}\,\inf_{v \in \Vv(G)} \|u - v\|_{L_2(\IIi,V,\mu)},
\end{equation}
and consequently,
\begin{equation} \label{|u - u_G|<}
\|u - u_G\|_{L_2(\IIi,V,\mu)}
\ \le \
\sqrt{\frac{R}{r}} \, \|u - \Ss_G u\|_{L_2(\IIi,V,\mu)}.
\end{equation}

\begin{theorem} 
Let the assumptions and the notation of Lemma \ref{lemma|u_bs|_W} hold.
Let furthermore $\bc = (c_j)_{j \in \NN}$ be any positive sequence such that 
$c_j > 1$, such that the sequence  $\bc^{-1} = (c_j^{-1})_{j \in \NN}$ belongs to $\ell_2(\NN)$
 and such that for the sequence 
\begin{equation} \label{seq[bb]}
\bb:= (b_j)_{j \in \NN}, \quad b_j:= c_j d_j,
\end{equation} 
there holds the condition
\begin{equation}\nonumber
\begin{cases} 
\|\bb\|_{\ell_1(\NN)} \ < \ 1,  \ & \ m = 1, 
\\[1ex]
\|\bb\|_{\ell_1(\NN)} \ \le \ 1,  \ & \ m > 1.
\end{cases}
\end{equation}
For any $T \ge 1$,  put $n := |E_{1,\bb}(T)|$; \ $\Vv_n:= \Vv\big(E_{1,\bb}(T)\big)$; \  
$\Pp_n:= \Ss_{E_{1,\bb}(T)}$; \  $u_n:= u_{E_{1,\bb}(T)}$. Then
$\Pp_n$ is the orthogonal projector
from $L_2(\IIi,V,\mu)$ onto the space $\Vv_n$ of dimension $n$,  and
\begin{equation}  \nonumber
\|u - u_n\|_{L_2(\IIi,V,\mu)}
\ \le \
\sqrt{\frac{R}{r}} \,
\big\|u - \Pp_n u\big\|_{L_2(\IIi,V,\mu)}
\ \le \ 
B\, n^{- 1/m},
\end{equation}
where
\[
B
:= \
4\pi \, C^{1/m}\, \sqrt{\frac{m R}{r}}\, K\, \|\bc^{-1}\|_{\ell_2(\FF)},
\] 
$C$ is the constant defined in \eqref{[constC]} for $a=1$ and $\bb$ as in \eqref{seq[bb]}.
\end{theorem}

\begin{proof}
By Lemma  \ref{lemma[u in A]} the solution  $u$ belongs to $A^{2,\bb}:= A^{2,\bb}(\TTm \times \IIi)$. Hence, 
by \eqref{|u - u_G|<}, Lemma \ref{lemma[|u|_LV<]}, Theorem \ref{theorem[d_n]-nonperiodic}  and \eqref{[|u|_A<]}
we have
\begin{equation} \nonumber
\begin{split}
\|u - u_n\|_{L_2(\IIi,V,\mu)} 
\ &\le \
\sqrt{\frac{R}{r}} \, \|u - \Pp_n(u)\|_{L_2(\IIi,V,\mu)} 
\ \le \
\sqrt{\frac{R}{r}} \, 2\pi \sqrt{m} \|u - \Pp_n(u)\|_{K^1(\TTm \times \IIi)}
\\[1ex]
\ &\le \
\sqrt{\frac{R}{r}} \, 4\pi \sqrt{m} \, C^{1/m} \|u\|_{A^{2,\bb}}\, n^{- 1/m}
\ \le \
\sqrt{\frac{R}{r}} \, 4\pi \sqrt{m}\, C^{1/m}\, K\, \|\bc^{-1}\|_{\ell_2(\FF)}\, n^{- 1/m}
\\[1ex]
\ &= \ 
B\, n^{- 1/m}.
\end{split}
\end{equation}
\hfill
\end{proof}

The following theorem can be proven in a similar way.

\begin{theorem} 
Let the assumptions and the notation of Lemma \ref{lemma|u_bs|_V} hold.
Let $\bc = (c_j)_{j \in \NN}$ be any positive sequence such that 
$c_j > 1$, such that the sequence  $\bc^{-1} = (c_j^{-1})_{j \in \NN}$ belongs to $\ell_2(\NN)$ 
 and such that for the sequence 
\begin{equation} \label{seq[bb](1)}
\bb:= (b_j)_{j \in \NN}, \quad b_j:= c_j d_j,
\end{equation} 
there holds the condition
\begin{equation}\nonumber
\begin{cases} 
\|\bb\|_{\ell_1(\NN)} \ < \ 1,  \ & \ m = 1, 
\\[1ex]
\|\bb\|_{\ell_1(\NN)} \ \le \ 1,  \ & \ m > 1.
\end{cases}
\end{equation}
For any $T \ge 1$,  put $n := |E_{1,\bb}(T)|$; \ $\Vv_n:= \Vv\big(E_{1,\bb}(T)\big)$; \  
$\Pp_n:= \Ss_{E_{1,\bb}(T)}$; \  $u_n:= u_{E_{1,\bb}(T)}$. Then
$\Pp_n$ is the orthogonal projector
from $L_2(\TTm \times \IIi)$ onto the space $\Vv_n$ of dimension $n$,  and
\begin{equation}  \nonumber
\big\|u - \Pp_n u\big\|_{L_2(\TTm \times \IIi)}
\ \le \ 
B\, n^{- 1/m},
\end{equation}
where
\[
B
:= \
4\pi \, C^{1/m}\,  K\, \|\bc^{-1}\|_{\ell_2(\FF)}
\] 
and $C$ is the constant defined in \eqref{[constC]} for $a=1$ and $\bb$ as in \eqref{seq[bb](1)}.
\end{theorem}

\section{Concluding remarks}
\label{conclusion}
We discussed the $\varepsilon$-dimension of certain Sobolev-analytic-type space which are characterized as anisotropic tensor products and arise the the regularity theory of parametric operator equations. The function space are tensor products of Sobolev-type function space defined on a finite dimensional domain and analytic function space defined on infinite dimensional domains.
The approach using the $\varepsilon$-dimension fixes a priori an approximation error and computes the number of linear information which is needed in an approximation method to obtain this fixed error. Such an analysis relies on delicate estimates on the cardinality of both finite and infinite-dimensional hyperbolic crosses. 
We established upper and lower bounds of the $\varepsilon$-dimension and Komogorov $n$-widths of our Sobolev-analytic-type function space which depend only on the smoothness differences in the finite dimensional Sobolev space and the finite dimension. This shows that asymptotically the costs of the infinite dimensional smooth approximation problem are dominated by the finite dimensional and less smooth conventional approximation problem. 
These index sets, we study here, might also arise in different applications and hence are of its own interest. We note that the methodology of the paper follows a strict guideline. We fix the error, we construct an index set which can realize this error and then, we have to compute the cardinality of that index set. Hence, this approach is fairly general and can also be applied in many more situations. In the present paper, as an example, the obtained results are applied to the Galerkin approximation of parametric elliptic PDEs.

\section*{Acknowledgements} Dinh Dung's research work is funded by Vietnam National Foundation for Science and Technology Development (NAFOSTED) under  Grant No. 102.01-2017.05. A part of this paper was done when
Dinh Dung and Vu Nhat Huy were working at  Advanced Study in Mathematics (VIASM). They  would like to thank  the VIASM  for providing a fruitful research environment and working condition.  Michael Griebel and Christian Rieger would like to thank the \emph{Deutsche Forschungsgemeinschaft (DFG)} for financial support through the CRC 1060, \emph{The Mathematics of Emergent Effects}. Furthermore, they would like to thank the \emph{ Hausdorff Center for Mathematics} for financial support.


\begin{thebibliography}{99}


\bibitem{BNTT}
J. Beck, F. Nobile, L. Tamellini, and R. Tempone, On the optimal polynomial approximation of stochastic PDEs by Galerkin and collocation methods,
Mathematical Models and Methods in Applied Sciences, 22(9), (2012), 1250023.

\bibitem{BR02}
O. Bratteli  and D. Robinson, 
\newblock Operator Algebras and Quantum Statistical Mechanics v.1, 
\newblock{\em Springer-Verlag}, 2002.

\bibitem{CCS}
A. Chkifa, A. Cohen, and C. Schwab, 
\newblock High-dimensional adaptive sparse polynomial interpolation and applications to parametric PDEs,
\newblock {\em Foundations of Computational Mathematics}, 14(4), (2014), 601-633. 

\bibitem{CDS10b}
A. Cohen, R. DeVore, and C. Schwab, 
\newblock Convergence rates of best N-term Galerkin approximations for a class of elliptic sPDEs, 
\newblock {\em Foundations of Computational Mathematics}, 10(6), (2010), 615-646.

\bibitem{CDS10a}
A. Cohen, R. DeVore, and C. Schwab,
\newblock Analytic regularity and polynomial approximation of parametric and stochastic elliptic PDEs, 
\newblock {\em Analysis and Applications}, 9(1), (2011), 11-47.

\bibitem{Di16}
D. D\~ung,
\newblock Linear collective collocation and  Galerkin approximations for parametric and stochastic elliptic PDEs, 
arXiv:1511.03377v5 [math.NA]. 

\bibitem{DG16}
D. D\~ung  and M. Griebel, 
\newblock Hyperbolic cross approximation in infinite dimensions, 
\newblock {\em Journal of Complexity}, 33(2016), 33-88. 

\bibitem{DU13}
D. D\~ung and T. Ullrich, 
\newblock N-widths and $\varepsilon$-dimensions for high-dimensional approximations,
\newblock{\em Foundations of Computational Mathematics} 13(2013), 965-1003.

\bibitem{HLP34}
G.H. Hardy, J.E. Littlewood and G. P\'olya,
\newblock Inequalities, 
\newblock {\em Cambridge University Press}, London 1934. 

\bibitem{Ko36}
A. Kolmogorov, \newblock
\"Uber die beste Ann\"aherung von Funktionen einer Funktionklasse,
\newblock {\em Annals of Mathematics}, {37}, (1936), 107-111.

\bibitem{KS16}
A. Kunoth and C. Schwab,\
\newblock Sparse adaptive tensor Galerkin approximations of stochastic
  pde-constrained control problems
\newblock {\em SIAM/ASA Journal on Uncertainty Quantification},
  4(1):1034--1059, 2016.

\bibitem{SG}
C. Schwab and C. J. Gittelson
\newblock Sparse tensor discretizations of high-dimensional parametric and stochastic PDEs
\newblock {\em Acta Numerica}, 20, (2011), 291-467

\bibitem{NW08}
E. Novak and H. Wo{\'z}niakowski,
Tractability of Multivariate Problems, Volume I: Linear Information,
EMS Tracts in Mathematics, Vol. 6,
Eur. Math. Soc. Publ. House, Z\"urich, 2008. 

\bibitem{NW10}
E. Novak and H. Wo{\'z}niakowski,
\newblock Tractability of Multivariate Problems, Volume II: Standard Information for Functionals
\newblock  EMS Tracts in Mathematics, Vol. 12,
\newblock  Eur. Math. Soc. Publ. House, Z\"urich, 2010. 

\bibitem{Ti60}
V. Tikhomirov,
\newblock Widths of sets in function spaces and the theory of best approximations,
\newblock{\em Uspekhi Matematicheskikh Nauk}, {15}(3), (93), (1960), 81-120.
English translation in 
\newblock {\em Russian Math. Survey}, 15, 1960.

\end{thebibliography}
\end{document}